\documentclass[a4paper, intlimits, reqno]{amsart}

\usepackage[english]{babel}
\usepackage[T1]{fontenc}
\usepackage[utf8]{inputenc}

\usepackage{amsmath}
\usepackage{amssymb}
\usepackage{MnSymbol}
\usepackage{amsthm}
\allowdisplaybreaks
\usepackage{amsfonts}
\usepackage{mathrsfs} 
\usepackage{mathtools}
\usepackage{nicefrac}
\usepackage{enumitem}
\usepackage{multicol}
\usepackage{url}
\usepackage{dsfont}
\usepackage[numbers,sort&compress]{natbib}
\usepackage{doi}
\usepackage{prettyref}
\usepackage{xcolor}
\usepackage{orcidlink}

\newrefformat{defn}{Definition \ref{#1}}
\newrefformat{rem}{Remark \ref{#1}}
\newrefformat{sect}{Section \ref{#1}}
\newrefformat{prop}{Proposition \ref{#1}}
\newrefformat{thm}{Theorem \ref{#1}}
\newrefformat{cor}{Corollary \ref{#1}}
\newrefformat{ex}{Example \ref{#1}}

\swapnumbers
\newtheoremstyle{dotless}{}{}{\itshape}{}{\bfseries}{}{}{}
\theoremstyle{dotless}
\theoremstyle{plain}
\newtheorem{thm}{Theorem}[section]
\newtheorem{lem}[thm]{Lemma}
\newtheorem{prop}[thm]{Proposition}
\newtheorem{cor}[thm]{Corollary}
\theoremstyle{definition}
\newtheorem{defn}[thm]{Definition}
\newtheorem{rem}[thm]{Remark}
\newtheorem{exa}[thm]{Example}
\newcommand{\N} {\mathbb{N}}
\newcommand{\R} {\mathbb{R}}
\newcommand{\C} {\mathbb{C}}
\newcommand{\e}{\mathrm{e}}

\DeclareMathOperator{\id}{id}
\DeclareMathOperator{\re}{Re}
\providecommand{\differential}{\mathrm{d}}
\renewcommand{\d}{\differential}
\newcommand\rlim{
\mathchoice{\vcenter{\hbox{${\scriptstyle{+}}$}}}
{\vcenter{\hbox{$\scriptstyle{+}$}}}
{\vcenter{\hbox{$\scriptscriptstyle{+}$}}}
{\vcenter{\hbox{$\scriptscriptstyle{+}$}}}}
\newcommand{\vertiii}[1]{{\left\vert\kern-0.25ex\left\vert\kern-0.25ex\left\vert #1 
    \right\vert\kern-0.25ex\right\vert\kern-0.25ex\right\vert}}

\begin{document}

\title[Sun dual theory for bi-continuous semigroups]{Sun dual theory for bi-continuous semigroups}
\author[K.~Kruse]{Karsten Kruse\,\orcidlink{0000-0003-1864-4915}}
\thanks{K.~Kruse acknowledges the support by the Deutsche Forschungsgemeinschaft (DFG) within the Research Training
 Group GRK 2583 ``Modeling, Simulation and Optimization of Fluid Dynamic Applications''.}
\address[KK]{University of Twente, Department of Applied Mathematics, P.O. Box 217, 7500 AE Enschede, The Netherlands, and Hamburg University of Technology, Institute of Mathematics, Am Schwarzenberg-Campus~3, 21073 Hamburg, Germany}
\email{k.kruse@utwente.nl}
\author[F.L.~Schwenninger]{Felix L. Schwenninger\,\orcidlink{0000-0002-2030-6504}}
\address[FLS]{University of Twente, Department of Applied Mathematics, P.O. Box 217, 7500 AE Enschede, The Netherlands}
\email{f.l.schwenninger@utwente.nl}



\date{\today}
\begin{abstract}
The sun dual space corresponding to a strongly continuous semigroup is a known concept when dealing with dual semigroups, which are in general only weak$^*$-continuous. In this paper we develop a corresponding theory for bi-continuous semigroups under mild assumptions on the involved locally convex topologies. We also discuss sun reflexivity and Favard spaces in this context, extending classical results by van Neerven. 
\end{abstract}
\maketitle

\section{Introduction}

Semigroup theory is a well-established tool in the abstract study of evolution equations. Classically, {\it strongly continuous} semigroups of boun\-ded linear operators on Banach spaces (also called {\it $C_{0}$-semigroups}) are considered, meaning that the semigroup is strongly continuous with respect to the norm topology. This, however, limits the applicability of the theory in spaces such as $\mathrm{C}_{\operatorname{b}}(\R^{n})$ or $L^{\infty}$, ruling out interesting examples arising from (partial) differential equations. This fact is underlined by Lotz's result \cite{lotz1985} asserting that any strongly continuous semigroup on Grothendieck spaces with the Dunford--Pettis property is automatically uniformly continuous. 

On the other hand, it has long been known that strong continuity fails to be preserved for the dual semigroup 
$\left(T'(t)\right)_{t\ge0}\coloneqq\left(T(t)'\right)_{t\ge0}$ in general, and merely translates into weak$^*$-continuity. Nevertheless, the strong continuity of the ``pre-semigroup'' $\left(T(t)\right)_{t\ge0}$ encodes enough structure to allow for a rich theory. Following first results in the early days of semigroup theory; by Phillips \cite{phillips1955}, Hille--Phillips \cite{hille1996}, de Leeuw \cite{deleeuw1960}, see also Butzer--Berens \cite{butzer1967}; 
intensified research on dual semigroups was conducted in the 1980s centred around a ``Dutch school'' in a series of papers such as by Cl\'ement, Diekmann, Gyllenberg, Heijmans and Thieme \cite{clement1987,clement1988,clement1989,clement1989a,diekmann1991}, de Pagter \cite{depagter1989}. The renewed interest in dual semigroups was partially driven by the interest from applications in e.g.~delay equations \cite{diekmannvangils1995}. At a peak of these developments van Neerven \cite{vanneerven1992} finally provided a general comprehensive treatment of the theory, together with many new results clarifying especially the topological aspects, see also \cite{vanneerven1991,vanneerven1990,vanneerven1991a}. Since then, the interest in dual semigroups which fail to be  strongly continuous remained, and we name particularly applications in mathematical neuroscience \cite{Gils2013,Spe2020}. 

The key concept to compensate for the lack of strong continuity of dual semigroups is the notion of the {\it sun dual space} and the related sun dual semigroup. More precisely, given a strongly continuous semigroup $\left(T(t)\right)_{t\ge0}$ on a Banach space $X$, the {\it sun dual space} $X^{\odot}$ consists of the elements $x'$ in the continuous dual $X'$ such that $\lim_{t\to0\rlim }T'(t)x'=x'$. As $X^{\odot}$ is closed and $T'(t)$-invariant, the restrictions of $T'(t)$ to $X^{\odot}$ define a strongly continuous semigroup $\left(T^{\odot}(t)\right)_{t\ge0}$ on $X^{\odot}$, an object which is in many facets superior to the dual semigroup. 

Note that this approach can be viewed as a way to regain symmetry in duality for continuity properties of the semigroup. While this holds trivially for reflexive spaces $X$ ---in which case $X^{\odot}=X'$---, it is not surprising that sun duality comes with an adapted notion of reflexivity, so-called {\it sun reflexivity} (or $\odot$-reflexivity), which depends on the semigroup under consideration. In particular, if $X$ is $\odot$-reflexive with respect to the semigroup $\left(T(t)\right)_{t\ge0}$, then 
$\left(T^{\odot\odot}(t)\right)_{t\ge0}$ can be identified with $\left(T(t)\right)_{t\ge0}$ via the canonical isomorphism 
$j\colon X\to \left(X^{\odot}\right)', x\mapsto  \left(x^{\odot}\mapsto \langle x^{\odot},x\rangle\right)$. That this framework indeed leads to a meaningful theory is also reflected by the existence of an Eberlein--Shmulyan type theorem due to van Neerven \cite{vanneerven1991}, and de Pagter's characterisation of sun reflexivity \cite{depagter1989}, which can be seen as a variant of Kakutani's theorem.

About ten years after this flourishing period of dual semigroups, K\"uhnemund \cite{kuehnemund2001,kuehnemund2003} conceptualized semigroups which only satisfy weaker continuity properties through the notion of {\it bi-continuous semigroups}. 
More precisely, the strong continuity was relaxed to hold with respect to a Hausdorff locally convex topology $\tau$ coarser than the norm topology on $X$. Under the additional conditions that $\tau$ is sequentially complete on norm-bounded sets and the dual space of $(X,\tau)$ is norming, an exponentially bounded semigroup $\left(T(t)\right)_{t\ge0}$ on $X$ is called $\tau$-bi-continuous if the trajectories $T(\cdot)x$ are $\tau$-strong continuous and locally sequentially $\tau$-equicontinuous on norm-bounded sets. Since the weak$^{*}$-topology shares these properties, dual semigroups naturally fall in this framework. Thus the question becomes how the construction of the sun dual can be seen in this light. With this paper we would like to answer this question and hence generalise existing results for strongly continuous semigroups in the presence of previously missing topological subtleties. 

The interest in bi-continuous semigroups goes beyond the above mentioned special case of dual semigroups, as they, for instance, naturally emerge in the study of evolution equations on spaces of bounded continuous functions, most prominently parabolic problems, see e.g.~Farkas--Lorenzi \cite{farkas2009}, Metafune--Pallara--Wacker \cite{metafune2002}. In the last decades the abstract theory of bi-continuous semigroups has been further developed and variants of the classical case have been established, such as for instance perturbation results; Farkas \cite{farkas2004a,farkas2004b}, approximation results; Albanese--Mangino \cite{albanese2004} and mean ergodic theorems; Albanese--Lorenzi--Manco \cite{albanese2011}. In \cite{farkas2011} Farkas defined a proper concept for a dual bi-continuous semigroup by considering a suitable subspace $X^{\circ}$ of $X'$. In particular, the restriction of the dual semigroup on $X^{\circ}$ is again a $\sigma(X^{\circ},X)$-bi-continuous dual semigroup under some additional topological assumptions.

In this work we develop a sun dual theory for bi-continuous semigroups and discuss its peculiarities with respect to properties of the present topologies. This generalises the classical case, i.e.~strongly continuous semigroups with respect to the norm topology; henceforth simply called ``strongly continuous''. 
Apart from the abstract interest in developing a sun dual framework for bi-continuous semigroups, one of our main motivations to provide such generalizations are open problems of the following kind: We aim to extend the following theorem for strongly continuous semigroups to bi-countinuous ones.

\begin{thm}[{\cite[Theorem 2.9, p.~152]{JacoSchwWint22}}]\label{thm:Baillon_refine}
Let $(X,\|\cdot\|)$ be a Banach space
and $(T(t))_{t\geq 0}$ a strongly continuous semigroup on $X$ with generator $(A,D(A))$. 
Then the following assertions are equivalent:
\begin{enumerate}
\item[(i)] $Fav(T)=D(A)$ and $(T(t))_{t\geq 0}$ satisfies the $\mathrm{C}$-maximal regularity property.
\item[(ii)] $A$ extends to a bounded operator from $X$ to $X$.
\end{enumerate}
\end{thm}

Here $Fav(T)$ denotes the \emph{Favard space} of $(T(t))_{t\geq 0}$ given by
\[
Fav(T)\coloneqq\{x\in X\;|\;\limsup_{t\to 0\rlim}\tfrac{1}{t}\|T(t)x-x\|<\infty\}
\]
and \emph{$\mathrm{C}$-maximal regularity} refers to the property that 
\[
t\mapsto \int_{0}^{t}T(t-s)f(s)\,\mathrm{d}s \in \mathrm{C}([0,\infty);D(A))\text{ for all }f\in \mathrm{C}([0,\infty);X).
\]
Note that \cite[Theorem 2.9, p.~152]{JacoSchwWint22} lists another equivalent condition, which relates to control theory, see also \cite[Remark 2.4, p.~148]{JacoSchwWint22}. Following this, the question whether \prettyref{thm:Baillon_refine} can be formulated for bi-continuous semigroups is relevant for studying 
generalizations of control theoretic notions in non-strongly continuous semigroup settings. 
The concept of sun dual spaces for strongly continuous semigroups is pivotal in the proof of the non-trivial implication  (i) $\Rightarrow$ (ii) in \prettyref{thm:Baillon_refine}. The argument is based, among other tools, on two characterizations due to van Neerven, \cite[Theorems 3.2.8, 3.2.9, p.~57]{vanneerven1992}: 
The first stating that an element $x\in X$ belongs to $Fav(T)$ 
if and only if that there exists a bounded sequence $(y_{n})_{n\in\N}$ in $X$ such that 
$\lim_{n\to\infty}R(\lambda,A)y_{n}=x$ for some (all) $\lambda$ in the resolvent set $\rho(A)$ of $A$ 
where $R(\lambda,A)\coloneqq(\lambda \id -A)^{-1}$. 
The second  result claims that the property $Fav(T)=D(A)$ is equivalent to the condition that 
$R(\lambda,A)B_{\|\cdot\|^{\odot}}$ is closed in $X$ for some (all) $\lambda\in\rho(A)$ where 
$B_{(X,\|\cdot\|^{\odot})}\coloneqq\{x\in X\colon \|x\|^{\odot}\leq1\}$ is the unit ball w.r.t.~the norm
\[
\|x\|^{\odot}\coloneqq\sup_{x^{\odot}\in X^{\odot},\|x^{\odot}\|_{X'}\leq 1}|\langle x^{\odot},x\rangle|,\quad x\in X,
\]
which is equivalent to $\|\cdot\|$ by \cite[Theorem 1.3.5, p.~7]{vanneerven1992}. 
As a stepping stone towards a bi-continuous version of \prettyref{thm:Baillon_refine}, in this work we provide 
counterparts of \cite[Theorems 3.2.8, 3.2.9, p.~57]{vanneerven1992} for bi-continuous semigroups 
in \prettyref{thm:favard} and \prettyref{thm:favard_gamma_closed_resolv}. 
However, to even conclude a bi-continuous variant of \prettyref{thm:Baillon_refine} from this, one would have to bypass an argument, which was based on results by Bourgain--Rosenthal in the case of strongly continuous semigroups in \cite[Theorem 2.9, p.~152]{JacoSchwWint22}. The study of this gap goes beyond the scope of this paper and is subject to future work.

Let us briefly highlight some of our findings in the following. 
Starting from Farkas' dual space \cite{farkas2011}
\[
X^{\circ}\coloneqq\{x'\in X'\;|\;x'\;\tau\text{-sequentially continuous on } \|\cdot\|\text{-bounded sets}\},
\]
which is a closed subspace of $X'$ and invariant under the dual semigroup, we define the \emph{bi-sun dual space} $X^{\bullet}$ as the space of strong continuity for the restricted dual semigroup $T^{\circ}(t)\coloneqq T'(t)|_{X^{\circ}}$, $t\ge0$.
Under the additional assumptions that 
\begin{enumerate}
\item $X^{\circ}=\overline{(X,\tau)'}^{\|\cdot\|_{X'}}$,  
\item $X^{\circ}\cap \{x'\in X'\;|\;\|x'\|_{X'}\leq 1\}$ is sequentially $\sigma(X^{\circ},X)$-complete, and that
\item every $\|\cdot\|_{X'} $-bounded $\sigma(X^{\circ},X)$-null sequence in $X^{\circ}$ is $\tau$-equicontinuous 
on $\|\cdot\|$-bounded sets,
\end{enumerate}
 we can subsequently show that the norm defined by
\begin{equation*}
	\|x\|^{\bullet}\coloneqq\sup_{x^{\bullet}\in X^{\bullet}, \|x^{\bullet}\|_{X'}\leq 1}|\langle x^{\bullet},x\rangle|,\quad x\in X,
\end{equation*}
is equivalent to $\|\cdot\|$. This result, \prettyref{thm:eqiv_norms}, naturally generalises the corresponding known fact for strongly continuous semigroups (see \cite[Theorem 1.3.5, p.~7]{vanneerven1992} and the discussion in the previous paragraph). Further, let us point out that the assumptions $(1)$--$(3)$ are fulfilled 
by \prettyref{thm:sufficient_ass_adj_bi_cont} if $(X,\gamma)$ is a sequentially complete $c_{0}$-barrelled Mazur space, e.g.~a sequentially complete Mackey--Mazur space, where $\gamma\coloneqq\gamma(\|\cdot\|,\tau)$ denotes the mixed topology of Wiweger \cite{wiweger1961}. 
We henceforth say that $X$ is \emph{$\bullet$-reflexive} with respect to the $\tau$-bi-continuous semigroup $(T(t))_{t\ge0}$ if the canonical map 
$j\colon X\to {X^{\bullet}}'$ given by 
\[
\langle j(x), x^{\bullet}\rangle\coloneqq\langle x^{\bullet},x\rangle, \quad x\in X, x^{\bullet}\in X^{\bullet},
\]  
maps the space of strong continuity $X_{\operatorname{cont}}$ onto $X^{\bullet\bullet}$. Given the latter property, we show that $j\colon X\to {X^{\bullet}}'$ is surjective if and only if the unit ball $B_{(X,\|\cdot\|^{\bullet})}=\{x\in X\colon \|x\|^{\bullet}\leq1\}$ is $\sigma(X,X^{\bullet})$-compact, see \prettyref{thm:bi_sun_reflexive_j_surjective}, implying that $Fav(T)=D(A)$ 
if one (thus both) of the assertions holds. In analogy to strongly continuous semigroups, we are able to show in \prettyref{thm:favard_gamma_closed_resolv} that the domain of the semigroup generator equals the Favard space if and only if the set $R(\lambda,A)B_{(X,\|\cdot\|^{\bullet})}$ 
is closed with respect to $\tau$. 
The main results are thoroughly laid out by various natural classes of examples.

The article is organized as follows. In the preparatory \prettyref{sect:notions} we set the stage by discussing the topological assumptions and recapping some basics on bi-continuous semigroups as well as integral notions in this context. With the level of detail we aim for making the presentation rather self-contained, especially for readers less familiar with bi-continuous semigroups. 
In Sections \ref{sect:dual_semigroups} and \ref{sect:sun_dual} we present our approach to dual semigroups of bi-continuous semigroups and the sun dual space, respectively. The short \prettyref{sect:reflexivity} discusses the notion of sun reflexivity in this generalised context and we finish with studying the relation of the obtained results to Favard spaces, \prettyref{sect:favard}. 

\section{Notions and preliminaries}
\label{sect:notions}

For a vector space $X$ over the field $\R$ or $\C$ with a Hausdorff locally convex topology $\tau$ 
we denote by $(X,\tau)'$ the topological linear dual space and just write $X'\coloneqq(X,\tau)'$ 
if $(X,\tau)$ is a Banach space. For two topologies $\tau_{1}$ and $\tau_{2}$ on a space $X$, 
we write $\tau_{1}\leq\tau_{2}$ if the topology $\tau_{1}$ is coarser than $\tau_{2}$. 
Further, we use the symbol $\mathcal{L}(X;Y)\coloneqq\mathcal{L}((X,\|\cdot\|_{X});(Y,\|\cdot\|_{Y}))$ 
for the space of continuous linear operators from a Banach space $(X,\|\cdot\|_{X})$ 
to a Banach space $(Y,\|\cdot\|_{Y})$ and denote by $\|\cdot\|_{\mathcal{L}(X;Y)}$ the operator norm on 
$\mathcal{L}(X;Y)$. If $X=Y$, we set $\mathcal{L}(X)\coloneqq\mathcal{L}(X;X)$.

In the following, the mixed topology, \cite[Section 2.1]{wiweger1961}, and the notion of a Saks space 
\cite[I.3.2 Definition, p.~27--28]{cooper1978} will be crucial.

\begin{defn}[{\cite[Definition 2.2, p.~3]{kruse_schwenninger2022}}]\label{defn:mixed_top_Saks}
Let $(X,\|\cdot\|)$ be a Banach space and $\tau$ a Hausdorff locally convex topology on $X$ that is coarser 
than the $\|\cdot\|$-topology $\tau_{\|\cdot\|}$.  Then
\begin{enumerate}
	\item[(a)] the \emph{mixed topology} $\gamma \coloneqq \gamma(\|\cdot\|,\tau)$ is
	the finest linear topology on $X$ that coincides with $\tau$ on $\|\cdot\|$-bounded sets and such that 
	$\tau\leq \gamma \leq \tau_{\|\cdot\|}$, 
    \item[(b)] the triple $(X,\|\cdot\|,\tau)$ is called a \emph{Saks space} if there exists a directed system 
    of seminorms $\mathcal{P}_{\tau}$ that generates the topology $\tau$ such that
\begin{equation}\label{eq:saks}
\|x\|=\sup_{p\in\mathcal{P}_{\tau}} p(x), \quad x\in X.
\end{equation}
\end{enumerate}
\end{defn}

The mixed topology $\gamma$ is Hausdorff locally convex and 
our definition is equivalent to the one from the literature \cite[Section 2.1]{wiweger1961} 
due to \cite[Lemmas 2.2.1, 2.2.2, p.~51]{wiweger1961}.

\begin{defn}[{\cite[Definition 2.2, p.~423]{kruse_seifert2022a}}]
We call a Saks space $(X,\|\cdot\|,\tau)$ \emph{sequentially complete} if $(X,\gamma)$ is sequentially complete.
\end{defn}

We recall the definition of the Pettis integral of a function with values in a Hausdorff locally convex space, 
which we need later on and extends the original definition for Banach-valued functions 
\cite[Definition 2.1, p.~280]{pettis1938}. 

\begin{defn}
Let $(X,\tau)$ be a Hausdorff locally convex space over the field $\mathbb{K}\coloneqq\R$ or $\C$, 
$\Omega\subset\R$ a measurable set w.r.t.~the Lebesgue measure $\lambda$ and $L^{1}(\Omega)$ the space of 
(equivalence classes of) absolutely Lebesgue integrable functions from $\Omega$ to $\mathbb{K}$. 
A function $f\colon \Omega \to X$ is called \emph{weakly measurable} if the scalar-valued function 
$\langle x', f\rangle\coloneqq x'\circ  f$ is Lebesgue measurable for all $x' \in (X,\tau)'$.  
A weakly measurable function is said to be \emph{weakly integrable} if $x' \circ f \in L^{1}(\Omega)$ 
for all $x' \in (X,\tau)'$. 
A function $f\colon \Omega\to X$ is called $\tau$-\emph{Pettis integrable on} $\Omega$ \emph{in} $X$
if it is weakly integrable and
\[
 \exists\; x_{\Omega}(f) \in X\; \forall x' \in (X,\tau)': 
 \langle x' , x_{\Omega}(f) \rangle = \int\limits_{\Omega} \langle x' , f(s) \rangle\d\lambda(s). 
\]  
In this case $x_{\Omega}(f)$ is unique due to $X$ being Hausdorff and we define the $\tau$-\emph{Pettis integral 
of} $f$ \emph{on} $\Omega$ \emph{in} $X$ by 
\[
 \int\limits_{\Omega} f(s)\d\lambda(s)\coloneqq x_{\Omega}(f).
\]	
If $\Omega$ is an interval $[a,b]$, $a\leq b$, we usually write
\[
\int\limits_{a}^{b} f(s)\d s\coloneqq\int\limits_{[a,b]} f(s)\d\lambda(s).
\]
\end{defn}

\begin{defn}
Let $(X,\|\cdot\|,\tau)$ be a sequentially complete Saks space and $\Omega\subset\R$ non-empty. 
We set 
\[
  \mathrm{C}_{\tau,\operatorname{b}}(\Omega;X)
\coloneqq\{f\in\mathrm{C}(\Omega;(X,\tau))\;|\;\|f\|_{\infty}\coloneqq\sup_{x\in\Omega}\|f(x)\|<\infty\}
\]
where $\mathrm{C}(\Omega;(X,\tau))$ is the space of continuous functions from $\Omega$ to $(X,\tau)$.
\end{defn}

\begin{prop}\label{prop:Riemann_Pettis}
Let $(X,\|\cdot\|,\tau)$ be a sequentially complete Saks space and $a,b\in\R$ with $a<b$. 
\begin{enumerate}
\item[(a)] If $f\in\mathrm{C}_{\tau,\operatorname{b}}([a,b];X)$, 
then $f$ is $\tau$-Riemann integrable, $\gamma$-Riemann integrable, $\tau$-Pettis integrable and 
$\gamma$-Pettis integrable on $[a,b]$ in $X$ and all four integrals coincide. 
\item[(b)] If $f\in\mathrm{C}_{\tau,\operatorname{b}}([a,\infty);X)$ is improper 
$\tau$-Riemann integrable on $[a,\infty)$ such that even $|\langle x',f\rangle|$ is improper Riemann integrable 
on $[a,\infty)$ for all $x'\in (X,\tau)'$, then $f$ is improper $\gamma$-Riemann integrable, 
$\tau$-Pettis integrable and $\gamma$-Pettis integrable on $[a,\infty)$ in $X$ and all four integrals coincide. 
\end{enumerate}
\end{prop}
\begin{proof}
(a) It is a direct consequence of the proof of \cite[Proposition 1.1, p.~232]{komatsu1964}, 
the sequential completeness of the Saks space $(X,\|\cdot\|,\tau)$ and \cite[Corollary 2.3.2, p.~55]{wiweger1961}
that $f$ is $\tau$-Riemann integrable. We note that $\langle x',f\rangle$ is continuous on $[a,b]$ and thus 
Borel-measurable for all $x'\in(X,\tau)'$ since $f$ is $\tau$-continuous. Further, the definition of the $\tau$-Riemann 
integral $\operatorname{R}$-$\int_{a}^{b}f(s)\d s\in X$ by Riemann sums implies that 
\[
\langle x',\operatorname{R}\text{-}\int_{a}^{b}f(s)\d s\rangle
=\int_{a}^{b}\langle x',f(s)\rangle\d s
=\int_{[a,b]}\langle x',f(s)\rangle\d\lambda(s)
\]
for all $x'\in(X,\tau)'$. Thus $f$ is $\tau$-Pettis integrable on $[a,b]$ in $X$ 
and the Riemann and the Pettis integral coincide. 

Furthermore, the $\tau$-Riemann integrability of $f$ implies that the Riemann sums are $\tau$-convergent. 
They are even $\|\cdot\|$-bounded as $f$ is $\|\cdot\|$-bounded. 
It follows from \cite[I.1.10 Proposition, p.~9]{cooper1978} that the Riemann sums are $\gamma$-convergent and their 
$\gamma$-limit coincides with their $\tau$-limit because $\gamma$ is stronger than $\tau$.
Thus $f$ is $\gamma$-Riemann integrable on $[a,b]$ in $X$ and this integral coincides with the $\tau$-Riemann 
integral. 

Now, we only need to prove that $f\colon [a,b]\to (X,\gamma)$ is continuous. 
Then it follows as above that $f$ is $\gamma$-Pettis integrable on $[a,b]$ in $X$ and that 
the Riemann and the Pettis integral coincide. 
Let $(x_{n})_{n\in\N}$ be a sequence in $[a,b]$ that converges to $x_{0}\in[a,b]$. 
Then the sequence $(f(x_{n}))_{n\in\N}$ converges to $f(x_{0})$ in $(X,\tau)$ and is $\|\cdot\|$-bounded since 
$f\in\mathrm{C}_{\tau,\operatorname{b}}([a,b];X)$.
By \cite[I.1.10 Proposition, p.~9]{cooper1978} it follows that $(f(x_{n}))_{n\in\N}$ converges to $f(x_{0})$ in $(X,\gamma)$, 
implying that $f\colon[a,b]\to (X,\gamma)$ is continuous. 

(b) The proof is analogous to (a). The condition that $|\langle x',f\rangle|$ is improper Riemann integrable 
on $[a,\infty)$ for all $x'\in (X,\tau)'$ guarantees that $\langle x',f\rangle$ is Lebesgue integrable on 
$[a,\infty)$ and 
\[
\int_{a}^{\infty}\langle x',f(s)\rangle\d s
=\int_{[a,\infty)}\langle x',f(s)\rangle\d\lambda(s)
\]
for all $x'\in (X,\tau)'$ by \cite[Satz 6.3, p.~153]{elstrodt2005}.
\end{proof}

Using that a triple $(X,\|\cdot\|,\tau)$ fulfils \cite[Assumptions 1, p.~206]{kuehnemund2003} 
if and only if it is a sequentially complete Saks space (see \cite[p.~423]{kruse_seifert2022a}), 
\prettyref{defn:mixed_top_Saks} (a), \cite[Proposition 3.6 (ii), p.~1137]{federico2020} in 
combination with \cite[2.4.1 Corollary, p.~56]{wiweger1961} 
and that a sequence in $X$ is $\gamma$-convergent if and only if it is 
$\tau$-convergent and $\|\cdot\|$-bounded by \cite[I.1.10 Proposition, p.~9]{cooper1978}, we 
may rephrase the definition \cite[Definition 3, p.~207]{kuehnemund2003} of a bi-continuous semigroup 
on $X$ in the following way. 

\begin{defn}\label{defn:bi_continuous}
Let $(X,\|\cdot\|,\tau)$ be a sequentially complete Saks space and $\gamma \coloneqq \gamma(\|\cdot\|,\tau)$. 
A family $(T(t))_{t\geq 0}$ in $\mathcal{L}(X)$ is called a \emph{bi-continuous semigroup} on $X$
if 
\begin{enumerate}
\item[(i)] $(T(t))_{t\geq 0}$ is a \emph{semigroup}, i.e.~$T(t+s)=T(t)T(s)$ and $T(0)=\id$ for all $t,s\geq 0$,
\item[(ii)] $(T(t))_{t\geq 0}$ is $\gamma$\emph{-strongly continuous}, 
i.e.~the map $T_{x}\colon[0,\infty)\to(X,\gamma)$, $T_{x}(t)\coloneqq T(t)x$, is continuous for all $x\in X$, 
\item[(iii)] $(T(t))_{t\geq 0}$ is \emph{locally sequentially $\gamma$-equicontinuous}, 
i.e.~for every sequence $(x_n)_{n\in\N}$ in $X$, $x\in X$ 
with $\gamma\text{-}\lim\limits_{n\to\infty} x_n = x$ it holds that
     \[
      \gamma\text{-}\lim_{n\to\infty} T(t)(x_n-x) = 0
     \]
locally uniformly for all $t\in [0,\infty)$.
\end{enumerate}
\end{defn}

If we want to emphasize the dependence on the Saks space, we say that $(T(t))_{t\geq 0}$ is a bi-continuous semigroup 
on $(X,\|\cdot\|,\tau)$. 
\cite[Proposition 3.6 (ii), p.~1137]{federico2020} in combination with \cite[2.4.1 Corollary, p.~56]{wiweger1961} 
gives that a bi-continuous semigroup $(T(t))_{t\geq 0}$ on $X$ is \emph{exponentially bounded} (of type $\omega$), 
i.e.~there exist $M\geq 1$ and $\omega\in\R$ such that $\|T(t)\|_{\mathcal{L}(X)}\leq M\e^{\omega t}$ for all $t\geq 0$,
and we call
\[
\omega_{0}\coloneqq\omega_{0}(T)\coloneqq\inf\{\omega\in\R\;|\; \exists\; M\geq 1\;\forall\;t\geq 0:\;
\|T(t)\|_{\mathcal{L}(X)}\leq M\e^{\omega t}\}
\]
its \emph{growth bound} (see \cite[p.~7]{kuehnemund2001}). 
Due to the exponential boundedness of a bi-continuous semigroup and 
\cite[I.1.10 Proposition, p.~9]{cooper1978} we also may rephrase the definition \cite[Definition 1.2.6, p.~7]{farkas2003} 
of the generator of a bi-continuous semigroup in terms of the mixed topology. 

\begin{defn}
Let $(X,\|\cdot\|,\tau)$ be a sequentially complete Saks space
and $(T(t))_{t\geq 0}$ a bi-continuous semigroup on $X$. The \emph{generator} $(A,D(A))$ 
is defined by
\begin{align*}
D(A)\coloneqq&\Bigl\{x\in X\;|\;\gamma\text{-}\lim_{t\to 0\rlim}\frac{T(t)x-x}{t}\;\text{exists in } X\Bigr\},\\
Ax\coloneqq&\gamma\text{-}\lim_{t\to 0\rlim}\frac{T(t)x-x}{t},\;x\in D(A).
\end{align*}
\end{defn}
 
We recall that an element $\lambda\in\C$ belongs to the \emph{resolvent set} $\rho(A)$ of the generator $(A,D(A))$ 
if $\lambda-A\colon D(A)\to X$ is bijective and the \emph{resolvent} 
$R(\lambda,A)\coloneqq (\lambda-A)^{-1}\coloneqq(\lambda\id -A)^{-1}\in\mathcal{L}(X)$. 
For a linear subspace $Y$ of $X$ we define the \emph{part} $A_{\mid Y}$ \emph{of} $A$ \emph{in} $Y$ by 
\begin{align*}
D(A_{\mid Y})\coloneqq&\{y\in D(A)\cap Y\;|\;Ay\in Y\},\\
A_{\mid Y}y\coloneqq&Ay,\;y\in D(A_{\mid Y}).
\end{align*}
Usually, it is required that $Y$ is a $\|\cdot\|$-closed subspace (or a Banach space norm-continuously embedded 
in $X$) which is $(T(t))_{t\geq 0}$-invariant (see \cite[Chap.~II, Definition, p.~60]{engel_nagel2000}),
but this is not needed just for the sake of the definition of $A_{\mid Y}$. 
With these definitions at hand, we recall the following properties of the generator of 
a bi-continuous semigroup given 
in \cite[Definition 9, Propositions 10, 11, Theorem 12, Corollary 13, p.~213--215]{kuehnemund2003}, 
which are summarised in \cite[Theorems 5.5, 5.6, p.~339--340]{budde2019}, and may be rephrased in terms 
of the mixed topology by \cite[I.1.10 Proposition, p.~9]{cooper1978} and \prettyref{prop:Riemann_Pettis} as well.

\begin{thm}\label{thm:generator}
Let $(X,\|\cdot\|,\tau)$ be a sequentially complete Saks space 
and $(T(t))_{t\geq 0}$ a bi-continuous semigroup on $X$ with generator $(A,D(A))$. 
Then the following assertions hold:
\begin{enumerate}
\item[(a)] The generator $(A,D(A))$ is \emph{sequentially $\gamma$-closed}, 
i.e.~whenever $(x_{n})_{n\in\N}$ is a sequence in 
$D(A)$ such that $\gamma\text{-}\lim_{n\to\infty}x_{n}=x$ and $\gamma\text{-}\lim_{n\to\infty}Ax_{n}=y$ 
for some $x,y\in X$, then $x\in D(A)$ and $Ax=y$.
\item[(b)] The domain $D(A)$ is \emph{sequentially $\gamma$-dense}, i.e.~for each $x\in X$ there exists a 
sequence $(x_{n})_{n\in\N}$ in $D(A)$ such that $\gamma\text{-}\lim_{n\to\infty}x_{n}=x$.
\item[(c)] For $x\in D(A)$ we have $T(t)x\in D(A)$ and $T(t)Ax=AT(t)x$ for all $t\geq 0$. 
\item[(d)] For $t>0$ and $x\in X$ we have 
\[
\int_{0}^{t}T(s)x\d s\in D(A)
 \quad\text{and }
A\int_{0}^{t}T(s)x\d s=T(t)x-x
\]
where the integrals are $\gamma$-Pettis integrals.
\item[(e)] For $\re\lambda>\omega_{0}$ we have $\lambda\in\rho(A)$ and 
\[
R(\lambda,A)x=\int_{0}^{\infty}\e^{-\lambda s}T(s)x\d s,\quad x\in X,
\]
where the integral is a $\gamma$-Pettis integral.
\item[(f)] For each $\omega>\omega_{0}$ there exists $M\geq 1$ such that 
\[
\|R(\lambda,A)^{k}\|_{\mathcal{L}(X)}\leq \frac{M}{(\re \lambda -\omega)^{k}}
\]
for all $k\in\N$ and $\re\lambda>\omega$, i.e.~the generator $(A,D(A))$ is a \emph{Hille--Yosida operator}.
\item[(g)] Let $X_{\operatorname{cont}}$ be the space of $\|\cdot\|$-strong continuity for $(T(t))_{t\geq 0}$, i.e. 
\[
X_{\operatorname{cont}}\coloneqq\{x\in X\;|\;\lim_{t\to 0\rlim}\|T(t)x-x\|=0\}.
\]
Then $X_{\operatorname{cont}}$ is a $\|\cdot\|$-closed, sequentially $\gamma$-dense, 
$(T(t))_{t\geq 0}$-invariant linear subspace 
of $X$. Moreover, $X_{\operatorname{cont}}=\overline{D(A)}^{\|\cdot\|}$ 
and $(T(t)_{\mid X_{\operatorname{cont}}})_{t\geq 0}$ is the $\|\cdot\|$-strongly continuous semigroup on 
$X_{\operatorname{cont}}$ generated by the part $A_{\mid X_{\operatorname{cont}}}$ of $A$ in 
$X_{\operatorname{cont}}$ and 
\[
D(A_{\mid X_{\operatorname{cont}}})=\{x\in D(A)\;|\;Ax\in X_{\operatorname{cont}}\}.
\]
\end{enumerate}
\end{thm}

We added in part (g) that $X_{\operatorname{cont}}$ is sequentially $\gamma$-dense in $X$, which is a consequence of (b).

\section{Dual bi-continuous semigroups}
\label{sect:dual_semigroups}

We start this section by recalling the definition of the dual semigroup on $X^{\circ}$ of a bi-continuous semigroup 
on $X$ given in \cite{farkas2011}, where for a Saks space $(X,\|\cdot\|,\tau)$ we set
\[
X^{\circ}\coloneqq\{x'\in X'\;|\;x'\;\tau\text{-sequentially continuous on } \|\cdot\|\text{-bounded sets}\}.
\]

\begin{rem}
Let $(X,\|\cdot\|,\tau)$ be a Saks space.
Then $X^{\circ}$
is a closed linear subspace of the norm dual $X'$ and hence a Banach space 
by \cite[Proposition 2.1, p.~314]{farkas2011}. We note that it is assumed in \cite[Proposition 2.1, p.~314]{farkas2011} 
that the Saks space $(X,\|\cdot\|,\tau)$ is sequentially complete (see \cite[Hypothesis A (ii), p.~310--311]{farkas2011}) 
but an inspection of its proof shows that this assumption is not needed. 
\end{rem}

If $(X,\|\cdot\|,\tau)$ is a sequentially complete Saks space and $(T(t))_{t\geq 0}$ a 
bi-continuous semigroup on $X$, then the dual map $T'(t)\coloneqq T(t)'$ belongs to $\mathcal{L}(X')$ 
and leaves $X^{\circ}$ invariant for every $t\geq 0$ by \cite[Proposition 2.3, p.~315]{farkas2011}. 
Thus the restriction of the dual semigroup $(T'(t))_{t\geq 0}$ to $X^{\circ}$ forms a semigroup 
$(T^{\circ}(t))_{t\geq 0}$ on $X^{\circ}$ given by
$T^{\circ}(t)x^{\circ}\coloneqq T'(t)x^{\circ}$ for $t\geq 0$ and $x^{\circ}\in X^{\circ}$. This semigroup 
is clearly exponentially bounded (w.r.t.~$\|\cdot\|_{\mathcal{L}(X^{\circ})}$) and 
$\sigma(X^{\circ},X)$-strongly continuous, which implies that it is 
$\gamma^{\circ}$-strongly continuous by \prettyref{defn:mixed_top_Saks} (a) where 
$\gamma^{\circ}\coloneqq\gamma(\|\cdot\|_{X^{\circ}},\sigma(X^{\circ},X))$ 
and $\|\cdot\|_{X^{\circ}}$ denotes the restriction of $\|\cdot\|_{X'}$ to $X^{\circ}$.
In order to get a bi-continuous semigroup on 
$(X^{\circ},\|\cdot\|_{X^{\circ}},\sigma(X^{\circ},X))$, this triple needs to be a 
sequentially complete Saks space and 
$(T^{\circ}(t))_{t\geq 0}$ has to be locally sequentially $\gamma^{\circ}$-equicontinuous. 
Obviously, $\sigma(X^{\circ},X)$ is a coarser Hausdorff locally convex topology on $X^{\circ}$ 
than $\|\cdot\|_{X^{\circ}}$. Setting
\[
p_{N}(x^{\circ})\coloneqq\sup_{x\in N }|\langle x^{\circ},x \rangle|,\quad x^{\circ}\in X^{\circ},
\]
for finite $N\subset B_{\|\cdot\|}$, we get a directed system of seminorms that generates 
the $\sigma(X^{\circ},X)$-topology with
\[
\|x^{\circ}\|_{X^{\circ}}=\sup\{p_{N}(x^{\circ})\;|\;N\subset B_{\|\cdot\|}\,\text{finite}\}
\]
for all $x^{\circ}\in X^{\circ}$. 
Therefore $(X^{\circ},\|\cdot\|_{X^{\circ}},\sigma(X^{\circ},X))$ is a Saks space. 
The sequential completeness of $(X^{\circ},\|\cdot\|_{X^{\circ}},\sigma(X^{\circ},X))$ 
is not automatically fulfilled 
(see \cite[Example 2.2, p.~314--315]{farkas2011}) and for the local sequential 
$\gamma^{\circ}$-equicontinuity of $(T^{\circ}(t))_{t\geq 0}$ we need an additional assumption 
as well (see \cite[Hypothesis B and C, p.~314--315]{farkas2011}).

\begin{defn}\label{defn:adj_bi_cont}
Let $(X,\|\cdot\|,\tau)$ be a triple such that $(X,\|\cdot\|)$ is a Banach space, 
$\tau$ a Hausdorff locally convex topology on $X$ that is coarser than the $\|\cdot\|$-topology. 
We call $(X,\|\cdot\|,\tau)$ \emph{dual-consistent}, in  short \emph{d-consistent}, if 
\begin{enumerate}
\item[(i)] $X^{\circ}\cap B_{\|\cdot\|_{X'}}$ is sequentially complete w.r.t.~$\sigma(X^{\circ},X)$ where 
$B_{\|\cdot\|_{X'}}\coloneqq\{x'\in X'\;|\;\|x'\|_{X'}\leq 1\}$,
\item[(ii)] every $\|\cdot\|_{X'} $-bounded $\sigma(X^{\circ},X)$-null sequence in $X^{\circ}$ is $\tau$-equicontinuous 
on $\|\cdot\|$-bounded sets. 
\end{enumerate}
\end{defn}

Condition (i) of \prettyref{defn:adj_bi_cont} guarantees that $(X^{\circ},\|\cdot\|_{X^{\circ}},\sigma(X^{\circ},X))$ is sequentially complete. Condition (ii) of \prettyref{defn:adj_bi_cont} 
gives the local sequential $\gamma^{\circ}$-equicontinuity of
$(T^{\circ}(t))_{t\geq 0}$. 

\begin{prop}[{\cite[Proposition 2.4, p.~315]{farkas2011}, \cite[Lemma 1, p.~6]{budde2021}}]
\label{prop:adj_bi_cont}
Let $(X,\|\cdot\|,\tau)$ be a sequentially complete d-consistent Saks space, 
$\|\cdot\|_{X^{\circ}}$ the restriction of $\|\cdot\|_{X'}$ to $X^{\circ}$
and $(T(t))_{t\geq 0}$ a bi-continuous semigroup on $X$ with generator $(A,D(A))$.
Then the following assertions hold:
\begin{enumerate}
\item[(a)] The triple $(X^{\circ},\|\cdot\|_{X^{\circ}},\sigma(X^{\circ},X))$ is a 
sequentially complete Saks space.
\item[(b)] The operators given by $T^{\circ}(t)x^{\circ}\coloneqq T'(t)x^{\circ}$ for $t\geq 0$ and $x^{\circ}\in X^{\circ}$
form a bi-continuous semigroup on $(X^{\circ},\|\cdot\|_{X^{\circ}},\sigma(X^{\circ},X))$ with generator 
$(A^{\circ},D(A^{\circ}))$ fulfilling
\[
D(A^{\circ})=\{x^{\circ}\in X^{\circ}\;|\;\exists\;y^{\circ}\in X^{\circ}\;\forall\;x\in D(A): 
\langle Ax,x^{\circ} \rangle=\langle x,y^{\circ} \rangle\},\quad A^{\circ}x^{\circ}=y^{\circ}.
\]
\end{enumerate}
\end{prop}

Next, we take a closer look at the space $X^{\circ}$ and its relation to the dual space $(X,\gamma)'$ where 
$\gamma$ is the mixed topology of $\|\cdot\|$ and $\tau$. Both spaces coincide if $(X,\gamma)$ is a Mazur space. 
This will be a quite helpful observation in the next sections.

\begin{defn}[{\cite[p.~40]{wilansky1981}}]
A Hausdorff locally convex space $(X,\vartheta)$ with scalar field $\mathbb{K}\coloneqq\R$ or $\C$ is called 
\emph{Mazur space} if 
\[
(X,\vartheta)'=\{x'\colon X\to\mathbb{K}\;|\;x'\;\text{linear and }\vartheta\text{-sequentially continuous}\}
=:X_{\operatorname{seq-}\vartheta}'.
\]
\end{defn}

In the special case that $\vartheta=\sigma(X',X)$ a Banach space $(X,\|\cdot\|)$ such that $(X',\sigma(X',X))$ is a
Mazur space is also called \emph{d-complete} \cite[p.~624]{kappeler1986} or a $\mu B$ \emph{space} \cite[p.~45]{wilansky1981} 
or having \emph{Mazur's property} \cite[p.~51]{leung1991}.

\begin{rem}\label{rem:bi_cont_dual}
Let $(X,\|\cdot\|,\tau)$ be a Saks space.
Then $(X,\gamma)'$ is a closed linear subspace of $X'$, in particular a Banach space, 
and 
\[
(X,\gamma)'=\overline{(X,\tau)'}^{\|\cdot\|_{X'}}
\]
by \cite[I.1.18 Proposition, p.~15]{cooper1978}. Furthermore, $X^{\circ}=X_{\operatorname{seq-}\gamma}'$
by \cite[I.1.10 Proposition, p.~9]{cooper1978} and since we always have 
$X_{\operatorname{seq-}\gamma}'\subset X_{\operatorname{seq-}\|\cdot\|}'=X'$.
Thus $(X,\gamma)$ is a Mazur space if and only if 
\[
X^{\circ}=(X,\gamma)'.
\]
\end{rem}

\begin{prop}\label{prop:mazur}
Let $(X,\|\cdot\|,\tau)$ be a Saks space. 
\begin{enumerate}
\item[(a)] If $(X,\tau)$ is a Mazur space and every $\tau$-convergent sequence $\|\cdot\|$-bounded, 
then $(X,\gamma)$ is also a Mazur space.
\item[(b)] If $(X,\gamma)$ is a \emph{C-sequential space}, i.e.~every convex sequentially open subset of $(X,\gamma)$ 
is already open (see \cite[p.~273]{snipes1973}), then $(X,\gamma)$ is a Mazur space.
\end{enumerate}
\end{prop}
\begin{proof}
Part (b) is a direct consequence of \cite[Theorem 7.4, p.~52]{wilansky1981}. Let us turn to part (a). 
If $x'\colon X\to \mathbb{K}$ is linear and $\gamma$-sequentially continuous, then it is $\tau$-sequentially continuous 
on $\|\cdot\|$-bounded sets by \cite[I.1.10 Proposition, p.~9]{cooper1978} and thus $\tau$-continu\-ous as 
$(X,\tau)$ is a Mazur space and every $\tau$-convergent sequence $\|\cdot\|$-bounded. But this implies that $x'$ 
is $\gamma$-continuous since $\tau$ is coarser than $\gamma$.
\end{proof}

Examples of C-sequential spaces $(X,\gamma)$ are given in \cite[3.19, 3.20 Remarks, p.~14--15]{kruse_schwenninger2022} 
and \cite[3.23 Corollary (c), p.~16]{kruse_schwenninger2022}. We fix the following definition 
for the rest of the paper.

\begin{defn}
We call a Saks space $(X,\|\cdot\|,\tau)$ a Mazur space if $(X,\gamma)$ is a Mazur space.
\end{defn}
 
Now, let us revisit \prettyref{defn:adj_bi_cont} and give sufficient conditions in terms of the mixed topology $\gamma$ 
when the conditions of this definition are fulfilled. For that purpose we recall that a Hausdorff locally convex space $(X,\vartheta)$
is called $c_{0}$\emph{-barrelled} if every $\sigma((X,\vartheta)',X)$-null sequence in $(X,\vartheta)'$ 
is $\vartheta$-equicontinuous (see \cite[p.~249]{jarchow1981}, or \cite[Definition, p.~353]{webb1968} where such spaces are called 
\emph{sequentially barrelled}).

\begin{thm}\label{thm:sufficient_ass_adj_bi_cont}
Let $(X,\|\cdot\|,\tau)$ be a sequentially complete Saks space 
and $X_{\gamma}'\coloneqq(X,\gamma)'$.
\begin{enumerate}
\item[(a)] Let $(X,\gamma)$ be a Mazur space. Then condition (i) of 
\prettyref{defn:adj_bi_cont} is fulfilled if and only 
if $(X_{\gamma}',\tau_{\operatorname{c}}(X_{\gamma}',(X,\|\cdot\|)))$ is sequentially complete 
where $\tau_{\operatorname{c}}(X_{\gamma}',(X,\|\cdot\|))$ is the \emph{topology of uniform convergence on compact subsets} of $(X,\|\cdot\|)$. 
\item[(b)] If $(X,\gamma)$ is a $c_{0}$-barrelled Mazur space, then $(X,\|\cdot\|,\tau)$ is d-consistent. 
In particular, if $(X,\gamma)$ is a Mackey--Mazur space, then $(X,\|\cdot\|,\tau)$ is d-consistent. 
\end{enumerate}
\end{thm}
\begin{proof} 
(a) We have $X_{\gamma}'=X^{\circ}$ by \prettyref{rem:bi_cont_dual} and so
the triple $(X_{\gamma}',\|\cdot\|_{X_{\gamma}'},\sigma(X_{\gamma}',X))$ is a Saks space 
by our considerations above \prettyref{defn:adj_bi_cont}. Our claim follows from 
\cite[2.3.2 Corollary, p.~55]{wiweger1961} since 
condition (i) of \prettyref{defn:adj_bi_cont} is equivalent to the sequential completeness 
of $(X^{\circ},\|\cdot\|_{X^{\circ}},\sigma(X^{\circ},X))$, and 
$\gamma^{\circ}=\tau_{\operatorname{c}}(X_{\gamma}',(X,\|\cdot\|))$ 
by \cite[3.22 Proposition (a), p.~16]{kruse_schwenninger2022}.

(b) From \cite[I.1.7 Corollary, p.~7]{cooper1978}, \cite[I.1.10 Proposition, p.~9]{cooper1978}
and $(X,\gamma)$ being a Mazur space, we deduce that condition (ii) of \prettyref{defn:adj_bi_cont} is equivalent 
to the condition that every $\gamma^{\circ}$-null sequence in $X_{\gamma}'$ is $\gamma$-equicontinuous. 
Since every $\gamma^{\circ}$-null sequence is a $\sigma(X_{\gamma}',X)$-null sequence, it follows from 
$(X,\gamma)$ being $c_{0}$-barrelled that condition (ii) of \prettyref{defn:adj_bi_cont} 
is satisfied. From \cite[Proposition 4.4, p.~354]{webb1968} we deduce that $(X_{\gamma}',\sigma(X_{\gamma}',X))$ 
is sequentially complete and thus condition (i) of \prettyref{defn:adj_bi_cont} is also fulfilled 
by part (a) if $(X,\gamma)$ is a $c_{0}$-barrelled Mazur space.

If $(X,\gamma)$ is a Mackey--Mazur space, then it is $c_{0}$-barrelled 
by \cite[Proposition 4.3, p.~354]{webb1968} because $(X,\gamma)$ is sequentially complete.
\end{proof}

Let us come to some examples of sequentially complete d-consistent Mazur--Saks spaces. 
First, we recall some notions from general topology. A completely regular space $\Omega$ is called \emph{$k_{\R}$-space} if any map $f\colon\Omega\to\R$ 
whose restriction to each compact $K\subset\Omega$ is continuous, is already continuous on $\Omega$ 
(see \cite[p.~487]{michael1973}). In particular, locally compact Hausdorff spaces clearly are Hausdorff $k_{\R}$-spaces. In addition \emph{Polish spaces}, i.e.~separably completely metrisable spaces, are Hausdorff $k_{\R}$-spaces by \cite[Proposition 11.5, p.~181]{james1999} and \cite[3.3.20, 3.3.21 Theorems, p.~152]{engelking1989}. We recall that a Hausdorff space $\Omega$ is called \emph{hemicompact} 
if there is a sequence $(K_{n})_{n\in\N}$ of compact sets in $\Omega$ 
such that for every compact set $K\subset\Omega$ there is $N\in\N$ such that $K\subset K_{N}$ 
(see \cite[Exercises 3.4.E, p.~165]{engelking1989}). 
For instance, $\sigma$-compact locally compact Hausdorff spaces are 
hemicompact Hausdorff $k_{\R}$-spaces by \cite[Exercises 3.8.C (b), p.~195]{engelking1989}. 
Further, there are hemicompact Hausdorff $k_{\R}$-spaces that are neither locally compact nor metrisible by 
\cite[p.~267]{warner1958}.

Second, let $\mathrm{C}_{\operatorname{b}}(\Omega)$ be the space of bounded continuous functions on a 
completely regular Hausdorff space $\Omega$ and  
\[
\|f\|_{\infty}\coloneqq\sup_{x\in\Omega}|f(x)|,\quad f\in \mathrm{C}_{\operatorname{b}}(\Omega).
\]
We denote by $\tau_{\operatorname{co}}$ the \emph{compact-open topology}, i.e.~the topology of uniform convergence 
on compact subsets of $\Omega$, which is induced by the directed system of seminorms 
$\mathcal{P}_{\tau_{\operatorname{co}}}$ given by 
\[
p_{K}(f)\coloneqq\sup_{x\in K}|f(x)|,\quad f\in \mathrm{C}_{\operatorname{b}}(\Omega),
\]
for compact $K\subset \Omega$. 

Let $\mathcal{V}$ denote the set of all non-negative bounded functions $\nu$ on $\Omega$ 
that vanish at infinity, i.e.~for every $\varepsilon>0$ the set $\{x\in\Omega\;|\;\nu(x)\geq\varepsilon\}$ is compact. 
Let $\beta_{0}$ be the Hausdorff locally convex topology on $\mathrm{C}_{\operatorname{b}}(\Omega)$ that is induced 
by the seminorms 
\[
|f|_{\nu}\coloneqq\sup_{x\in\Omega}|f(x)|\nu(x),\quad f\in\mathrm{C}_{\operatorname{b}}(\Omega),
\]
for $\nu\in\mathcal{V}$. Due to \cite[Theorem 2.4, p.~316]{sentilles1972} 
we have $\gamma(\|\cdot\|_{\infty},\tau_{\operatorname{co}})=\beta_{0}$. 
Let $\mathrm{M}_{\operatorname{t}}(\Omega)$ denote the space of bounded Radon measures 
on a completely regular Hausdorff space $\Omega$ and $\|\cdot\|_{\mathrm{M}_{\operatorname{t}}(\Omega)}$ be the 
total variation norm (see e.g.~\cite[p.~439--440]{kunze2011} where $\mathrm{M}_{\operatorname{t}}(\Omega)$ is called 
$\mathcal{M}_{0}(\Omega)$). By \cite[Theorem 4.4, p.~320]{sentilles1972} it holds 
$\mathrm{M}_{\operatorname{t}}(\Omega)=(\mathrm{C}_{\operatorname{b}}(\Omega),\beta_{0})'$. 

Furthermore, a Banach space $(X,\|\cdot\|)$ is called \emph{weakly compactly generated (WCG)} 
if there is a $\sigma(X,X')$-compact set $K\subset X$ such that $X=\overline{\operatorname{span}}(K)$ where 
$\overline{\operatorname{span}}(K)$ denotes the $\|\cdot\|$-closure of $\operatorname{span}(K)$
(see \cite[Definition 13.1, p.~575]{fabian2011}).
A Banach space $(X,\|\cdot\|)$ is called \emph{strongly weakly compactly generated space (SWCG)} if there exists 
a $\sigma(X,X')$-compact set $K\subset X$ such that for every $\sigma(X,X')$-compact set $L\subset X$ 
and $\varepsilon>0$ there is $n\in\N$ with $L\subset (nK+\varepsilon B_{\|\cdot\|})$ by
\cite[p.~387]{schluechtermann1988}. In particular, every SWCG space is a WCG space by 
\cite[Theorem 2.5, p.~390]{schluechtermann1988}.
Examples of SWCG spaces are reflexive Banach spaces, separable \emph{Schur spaces} (i.e.~weakly convergent sequences 
are convergent \cite[p.~253]{fabian2011}), 
the space $\mathcal{N}(H)$ of trace class operators for a separable Hilbert space $H$ and the space 
$L^{1}(\Omega,\nu)$ w.r.t.~a $\sigma$-finite measure $\nu$ 
by \cite[2.3 Examples, p.~389--390]{schluechtermann1988}. Further examples of WCG spaces are 
separable Banach spaces and the space $c_{0}(\Gamma)$ of all real (or complex) valued bounded functions on 
a non-empty set $\Gamma$ that vanish at infinity by \cite[Examples, p.~575--576]{fabian2011}. 
The spaces $\ell^{\infty}$ and $\ell^{1}(\Gamma)$ for an uncountable set $\Gamma$ are not WCG 
by \cite[Examples (iv), p.~576]{fabian2011} and there exist examples of WCG spaces that are not SWCG 
by \cite[2.6 Example, p.~391]{schluechtermann1988}.
Moreover, we recall that a Banach space $(X,\|\cdot\|)$ has an \emph{almost shrinking basis} 
if it has a Schauder basis such that its associated sequence of coefficient functionals forms a Schauder basis 
of $(X',\mu(X',X))$ where $\mu(X',X)$ is the Mackey topology on $X'$ (see \cite[p.~75]{kalton1973}).

\begin{exa}\label{ex:mackey_mazur_examples}
\begin{enumerate}
\item[(a)] Let $(X,\|\cdot\|)$ be a Banach space and $\tau_{\|\cdot\|}$ the $\|\cdot\|$-topology. 
Then $\gamma(\|\cdot\|,\tau_{\|\cdot\|})=\tau_{\|\cdot\|}$ by \prettyref{defn:mixed_top_Saks} (a), 
the barrelled space $(X,\tau_{\|\cdot\|})$ is a C-sequential Mackey space, in particular Mazur. 
Thus $(X,\|\cdot\|,\tau_{\|\cdot\|})$ is a sequentially complete d-consistent Mazur--Saks space
by \prettyref{thm:sufficient_ass_adj_bi_cont} (b).
\item[(b)] Let $\Omega$ be a hemicompact Hausdorff $k_{\R}$-space, or a Polish space. 
Then 
\[
 \gamma(\|\cdot\|_{\infty},\tau_{\operatorname{co}})
=\beta_{0}
=\mu(\mathrm{C}_{\operatorname{b}}(\Omega),\mathrm{M}_{\operatorname{t}}(\Omega))
\]
and $(\mathrm{C}_{\operatorname{b}}(\Omega),\beta_{0})$ is a C-sequential Mackey space by 
\cite[3.20 Remark (a), p.~15]{kruse_schwenninger2022}. 
Hence $(\mathrm{C}_{\operatorname{b}}(\Omega),\|\cdot\|_{\infty},\tau_{\operatorname{co}})$ 
is a sequentially complete Mazur--Saks space by \cite[p.~19]{kruse_schwenninger2022} 
and d-consistent by \prettyref{thm:sufficient_ass_adj_bi_cont} (b).
\item[(c)] Let $H^{\infty}$ denote the Hardy space of bounded holomorphic functions on the open unit disc $\mathbb{D}\subset\C$ and 
$\beta_{1}\coloneqq \gamma({\|\cdot\|_{\infty}}_{\mid H^{\infty}},\tau_{\|\cdot\|_{1}})$ where $\tau_{\|\cdot\|_{1}}$ is the 
topology induced by the norm $\|\cdot\|_{1}$ given by 
\[
\|f\|_{1}\coloneqq \sup_{0<r<1}\int_{0}^{2\pi}|f(r\e^{\mathrm{i}\theta})|\d\theta,\quad f\in H^{\infty}.
\]
Then $(H^{\infty},\beta_{1})$ is a C-sequential Mackey space by \cite[V.2.14 Corollary, p.~239]{cooper1978} and 
\cite[Proposition 5.7, p.~2681--2682]{kruse_meichnser_seifert2018}. 
Thus $(H^{\infty},{\|\cdot\|_{\infty}}_{\mid H^{\infty}},\tau_{\|\cdot\|_{1}})$ is a sequentially complete Mazur--Saks space 
by \cite[V.2.5 Proposition, p.~234]{cooper1978} and d-consistent by \prettyref{thm:sufficient_ass_adj_bi_cont} (b). 
\item[(d)] Let $(X_{0},\|\cdot\|_{0})$ be a WCG-Schur space. 
Then $(X_{0}',\|\cdot\|_{X_{0}'},\sigma(X_{0}',X_{0}))$ is a sequentially complete Saks space and 
$\gamma(\|\cdot\|_{X_{0}'},\sigma(X_{0}',X_{0}))=\tau_{\operatorname{c}}(X_{0}',X_{0})$ by 
\cite[Example E), p.~66]{wiweger1961} and \cite[p.~21]{kruse_schwenninger2022}. 
Since $(X_{0},\|\cdot\|_{0})$ is a WCG space, $(X_{0}',\sigma(X_{0}',X_{0}))$ is a Mazur space by 
\cite[Corollary 3.5, p.~46]{wilansky1981}. It follows that $(X_{0}',\tau_{\operatorname{c}}(X_{0}',X_{0}))$ is a Mazur space 
by \prettyref{prop:mazur} (a) because every $\sigma(X_{0}',X_{0})$-convergent sequence is $\|\cdot\|_{X_{0}'}$-bounded 
by the uniform boundedness principle. It is also a Mackey space by \cite[Theorem 3.2, p.~85]{martinpeinador2015}, 
in particular $\tau_{\operatorname{c}}(X_{0}',X_{0})=\mu(X_{0}',X_{0})$, because $(X_{0},\|\cdot\|_{0})$ is a Schur space. 
We deduce from \prettyref{thm:sufficient_ass_adj_bi_cont} (b) 
that $(X_{0}',\|\cdot\|_{X_{0}'},\sigma(X_{0}',X_{0}))$ is d-consistent.
\item[(e)] Let $(X_{0},\|\cdot\|_{0})$ be a Banach space and 
\begin{enumerate}
\item[(i)] $(X_{0},\|\cdot\|_{0})$ be an SWCG space, or
\item[(ii)] $(X_{0},\|\cdot\|_{0})$ have an almost shrinking basis and let $(X_{0},\sigma(X_{0},X_{0}'))$ be sequentially complete.
\end{enumerate}
Then $\gamma(\|\cdot\|_{X_{0}'},\mu(X_{0}',X_{0}))=\mu(X_{0}',X_{0})$ and $(X_{0}',\mu(X_{0}',X_{0}))$ 
is a C-sequen\-tial Mackey space in both cases by \cite[3.19 Remark (c), p.~14]{kruse_schwenninger2022} 
and \cite[3.20 Remark (c), p.~15]{kruse_schwenninger2022}. 
Therefore $(X_{0}',\|\cdot\|_{X_{0}'},\mu(X_{0}',X_{0}))$ is a sequentially complete Mazur--Saks space by 
\cite[p.~22]{kruse_schwenninger2022} and d-consistent by \prettyref{thm:sufficient_ass_adj_bi_cont} (b).
\item[(f)] Let $H$ be a separable Hilbert space and $\mathcal{N}(H)$ the space of trace class operators 
in $\mathcal{L}(H)=\mathcal{N}(H)'$.  
Let $\tau_{\operatorname{sot}^{\ast}}$ be the symmetric
strong operator topology, i.e.~the Hausdorff locally convex topology on $\mathcal{L}(H)$ 
generated by the directed system of seminorms
\[
p_{N}(R)\coloneqq\max\bigl(\sup_{x\in N}\|Rx\|_{H},\sup_{x\in N}\|R^{\ast}x\|_{H}\bigr),\quad R\in \mathcal{L}(H),
\]
for finite $N\subset H$ where $R^{\ast}$ is the adjoint of $R$. 
We denote by $\beta_{\operatorname{sot}^{\ast}}$ the mixed topology 
$\gamma(\|\cdot\|_{\mathcal{L}(H)},\tau_{\operatorname{sot}^{\ast}})$. 
The triple $(\mathcal{L}(H),\|\cdot\|_{\mathcal{L}(H)},\tau_{\operatorname{sot}^{\ast}})$ 
is a sequentially complete Saks space, $\beta_{\operatorname{sot}^{\ast}}=\mu(\mathcal{L}(H),\mathcal{N}(H))$ 
and $(\mathcal{L}(H),\beta_{\operatorname{sot}^{\ast}})$ is a C-sequential Mackey space
by \cite[4.12 Example, p.~24--25]{kruse_schwenninger2022}. 
We derive from \prettyref{thm:sufficient_ass_adj_bi_cont} (b) that 
$(\mathcal{L}(H),\|\cdot\|_{\mathcal{L}(H)},\tau_{\operatorname{sot}^{\ast}})$ is d-consistent.
\end{enumerate} 
\end{exa}

That $(X,\|\cdot\|,\tau_{\|\cdot\|})$ in \prettyref{ex:mackey_mazur_examples} (a) 
is a sequentially complete d-consistent Mazur--Saks space is well-known 
(see \cite[Proposition 3.18, p.~78]{kuehnemund2001}). 
That the triple $(\mathrm{C}_{\operatorname{b}}(\Omega),\|\cdot\|_{\infty},\tau_{\operatorname{co}})$ 
in \prettyref{ex:mackey_mazur_examples} (b) is 
a sequentially complete d-consistent Mazur--Saks space
is contained in \cite[p.~318]{farkas2011} if $\Omega$ is a $\sigma$-compact locally compact Hausdorff space, 
or a Polish space. 
In regard to example (c) we note that the space $H^{\infty}$ equipped with the induced mixed topology 
$\widetilde{\beta}_{0}\coloneqq \gamma({\|\cdot\|_{\infty}}_{\mid H^{\infty}},{\tau_{\operatorname{co}}}_{\mid H^{\infty}})
={\beta_{0}}_{\mid H^{\infty}}$ 
by \cite[I.4.6 Proposition, p.~44]{cooper1978} is not a Mackey space by \cite[V.2.7 Corollary, p.~235]{cooper1978}. 
Concerning example (e), there are spaces which fulfil condition (ii) but not condition (i) by 
\cite[Example 2.6, p.~391]{schluechtermann1988} and \cite[p.~15]{kruse_schwenninger2022}.

\section{Sun duals for bi-continuous semigroups}
\label{sect:sun_dual}

Let us consider a special case of \prettyref{prop:adj_bi_cont}, 
namely, \prettyref{ex:mackey_mazur_examples} (a).
Let $(X,\|\cdot\|)$ be a Banach space and $(T(t))_{t\geq 0}$ a $\|\cdot\|$-strongly continuous semigroup on 
$X$. Choosing $\tau$ as the $\|\cdot\|$-topology, we deduce that $X^{\circ}=X'$, $\sigma(X^{\circ},X)=\sigma(X',X)$ 
and that the dual semigroup $(T^{\circ}(t)=T'(t))_{t\geq 0}$ on $X'$ is 
bi-continuous on $(X',\|\cdot\|_{X'},\sigma(X',X))$ (cf.~\cite[Proposition 3.18, p.~78]{kuehnemund2001}). 
For such a semigroup the notion of the \emph{sun dual} $X^{\odot}$ was introduced (see \cite[p.~5]{vanneerven1992}), 
namely, the subspace of $X'$ on which the dual semigroup acts $\|\cdot\|_{X'}$-strongly, i.e.\
\[
X^{\odot}\coloneqq\{x'\in X'\;|\;\lim_{t\to 0\rlim}\|T'(t)x'-x'\|_{X'}=0\}.
\]
The generator $(A^{\odot},D(A^{\odot}))$ of the restriction 
$(T^{\odot}(t))_{t\geq 0}\coloneqq(T'(t)_{\mid X^{\odot}})_{t\geq 0}$ is the part of $A'$ in $X^{\odot}$ by 
\cite[Theorem 1.3.3 p.~6]{vanneerven1992}.
We generalise this notion to the semigroup $(T^{\circ}(t))_{t\geq 0}$ from 
\prettyref{prop:adj_bi_cont}, so we introduce the subspace of $X^{\circ}$ 
on which the semigroup $(T^{\circ}(t))_{t\geq 0}$ acts 
$\|\cdot\|_{X'}$-strongly and get the following corollary, which generalises 
\cite[Theorems 1.3.1, 1.3.3 p.~5--6]{vanneerven1992}.

\begin{cor}\label{cor:bi_sun_dual}
Let $(X,\|\cdot\|,\tau)$ be a sequentially complete d-consistent Saks space, 
$\|\cdot\|_{X^{\circ}}$ the restriction of $\|\cdot\|_{X'}$ to $X^{\circ}$  
and $(T(t))_{t\geq 0}$ a bi-continuous semigroup on $X$ with generator $(A,D(A))$. 
We define the \emph{bi-sun dual}
\[
X^{\bullet}\coloneqq\{x^{\circ}\in X^{\circ}\;|\;\lim_{t\to 0\rlim}\|T^{\circ}(t)x^{\circ}-x^{\circ}\|_{X'}=0\}.
\]
Then $X^{\bullet}$ is a $\|\cdot\|_{X^{\circ}}$-closed, 
sequentially $\gamma(\|\cdot\|_{X^{\circ}},\sigma(X^{\circ},X))$-dense, 
$(T^{\circ}(t))_{t\geq 0}$-invariant subspace of $X^{\circ}$. 
Further, $X^{\bullet}=\overline{D(A^{\circ})}^{\|\cdot\|_{X^{\circ}}}$ and 
$(T^{\bullet}(t))_{t\geq 0}\coloneqq(T^{\circ}(t)_{\mid X^{\bullet}})_{t\geq 0}$ is the $\|\cdot\|_{X^{\circ}}$-strongly 
continuous semigroup on $X^{\bullet}$ generated by the part $A^{\bullet}$ of $A^{\circ}$ 
in $X^{\bullet}$ as well as $\omega_ {0}(T^{\bullet})\leq\omega_{0}(T)$.
\end{cor}
\begin{proof}
We only need to prove $\omega_ {0}(T^{\bullet})\leq\omega_{0}(T)$. The rest of the corollary is a direct 
consequence of \prettyref{thm:generator} (g) and \prettyref{prop:adj_bi_cont}. 
We note that 
\begin{align*}
  \|T^{\bullet}(t)\|_{\mathcal{L}(X^{\bullet})}
&=\sup_{\substack{x^{\bullet}\in X^{\bullet}\\ \|x^{\bullet}\|_{X'}\leq 1}}\|T^{\bullet}(t)x^{\bullet}\|_{X'}
 =\sup_{\substack{x^{\bullet}\in X^{\bullet}\\ \|x^{\bullet}\|_{X'}\leq 1}}
  \sup_{\substack{x\in X\\ \|x\|\leq 1}}|\langle T^{\bullet}(t)x^{\bullet},x\rangle|\\
&=\sup_{\substack{x^{\bullet}\in X^{\bullet}\\ \|x^{\bullet}\|_{X'}\leq 1}}
  \sup_{\substack{x\in X\\ \|x\|\leq 1}}|\langle x^{\bullet},T(t)x\rangle|
 \leq\sup_{\substack{x^{\bullet}\in X'\\ \|x^{\bullet}\|_{X'}\leq 1}}
  \sup_{\substack{x\in X\\ \|x\|\leq 1}}|\langle x^{\bullet},T(t)x\rangle|\\
&=\sup_{\substack{x\in X\\ \|x\|\leq 1}}\|T(t)x\|
 =\|T(t)\|_{\mathcal{L}(X)}
\end{align*}
for all $t\geq 0$, yielding $\omega_ {0}(T^{\bullet})\leq\omega_{0}(T)$.
\end{proof}

\begin{rem}
Let $(X,\|\cdot\|)$ be a Banach space. For a $\|\cdot\|$-strongly continuous semigroup $(T(t))_{t\geq 0}$ on $X$ 
we have $X^{\bullet}=X^{\odot}$, $(T^{\bullet}(t))_{t\geq 0}=(T^{\odot}(t))_{t\geq 0}$ and $A^{\bullet}=A^{\odot}$.
\end{rem} 

Let us turn to a generalisation of \cite[Theorem 1.3.5, p.~7]{vanneerven1992} 
(see also \cite[Theorem 14.2.1, p.~422--423]{hille1996}).

\begin{thm}\label{thm:eqiv_norms}
Let $(X,\|\cdot\|,\tau)$ be a sequentially complete d-consistent Mazur--Saks space 
and $(T(t))_{t\geq 0}$ a bi-continuous semigroup on $X$. We set  
\[
\|x\|^{\bullet}\coloneqq\sup_{\substack{x^{\bullet}\in X^{\bullet}\\ \|x^{\bullet}\|_{X'}\leq 1}}
                 |\langle x^{\bullet},x\rangle|,\quad x\in X.
\]
Then $\|\cdot\|^{\bullet}$ and $\|\cdot\|$ are equivalent norms on $X$.
\end{thm}
\begin{proof}
It follows from the definition that $\|x\|^{\bullet}\leq \|x\|$ for all $x\in X$. For the converse estimate 
let $\varepsilon>0$ and $x\in X$. We choose $M\geq 0$ such that 
$\sup_{t\in[0,\delta)}\|T(t)\|_{\mathcal{L}(X)}\leq M$ for 
some $\delta>0$ by the exponential boundedness of $(T(t))_{t\geq 0}$. 
Let $\mathcal{P}_{\gamma}$ be a directed system of seminorms that generates the mixed topology 
$\gamma=\gamma(\|\cdot\|,\tau)$. 
For $p_{\gamma}\in \mathcal{P}_{\gamma}$ there is $x^{\circ}\in X^{\circ}=(X,\gamma)'$ such that 
$\langle x^{\circ},x\rangle=p_{\gamma}(x)$ 
and $|\langle x^{\circ},z\rangle|\leq p_{\gamma}(z)$ for all $z\in X$ by \prettyref{rem:bi_cont_dual}
and the Hahn--Banach theorem. For any $x^{\circ}\in (X,\gamma)'$ and $t>0$ we observe that the map $s\mapsto T(s)x$
is $\gamma$-Pettis integrable on $[0,t]$ by \prettyref{thm:generator} (d) and 
\begin{equation}\label{eq:eqiv_norms}
     \Bigl|\langle x^{\circ}, \frac{1}{t}\int_{0}^{t}T(s)x\d s-x\rangle\Bigr|
\leq \frac{1}{t}\int_{0}^{t}|\langle x^{\circ}, T(s)x-x\rangle|\d s
\leq \sup_{s\in [0,t]}|\langle x^{\circ},T(s)x-x\rangle|,
\end{equation}
which yields 
\[
p_{\gamma}\Bigl(\frac{1}{t}\int_{0}^{t}T(s)x\d s-x\Bigr)\leq \sup_{s\in [0,t]}p_{\gamma}(T(s)x-x)
\]
for any $p_{\gamma}\in P_{\gamma}$ by \cite[Proposition 22.14, p.~256]{meisevogt1997}. Hence it follows from 
the $\gamma$-strong continuity of $(T(t))_{t\geq 0}$ that 
$\gamma$-$\lim_{t\to 0\rlim}\frac{1}{t}\int_{0}^{t}T(s)x\d s-x=0$. 
Thus for $p_{\gamma}$ there is some $0<t_{0}<\delta$ such that 
$p_{\gamma}(\frac{1}{t_{0}}\int_{0}^{t_{0}}T(s)x\d s-x)\leq \varepsilon p_{\gamma}(x)$. We deduce that
\begin{align*}
  \Bigl|\langle \frac{1}{t_{0}}\int_{0}^{t_{0}}T^{\circ}(s)x^{\circ}\d s,x\rangle\Bigr|
&=\Bigl|\langle x^{\circ}, \frac{1}{t_{0}}\int_{0}^{t_{0}}T(s)x\d s\rangle\Bigr|
 \geq |\langle x^{\circ},x\rangle|-\Bigl|\langle x^{\circ}, \frac{1}{t_{0}}\int_{0}^{t_{0}}T(s)x\d s-x\rangle\Bigr|\\
&\geq p_{\gamma}(x)- p_{\gamma}\Bigl(\frac{1}{t_{0}}\int_{0}^{t_{0}}T(s)x\d s-x\Bigr)
 \geq (1-\varepsilon) p_{\gamma}(x).
\end{align*}
We have $\frac{1}{t_{0}}\int_{0}^{t_{0}}T^{\circ}(s)x^{\circ}\d s\in D(A^{\circ})\subset X^{\bullet}$ 
by \prettyref{thm:generator} (d), \prettyref{prop:adj_bi_cont} (b) and \prettyref{cor:bi_sun_dual} and we note 
that $\|\frac{1}{t_{0}}\int_{0}^{t_{0}}T^{\circ}(s)x^{\circ}\d s\|_{X'}\leq M$, which implies that
$\|x\|^{\bullet}\geq M^{-1}(1-\varepsilon) p_{\gamma}(x)$. As $\varepsilon>0$ is arbitrary, we obtain
\[
\|x\|^{\bullet}\geq M^{-1} p_{\gamma}(x).
\]
By \cite[Lemma 5.5 (a), p.~2680]{kruse_meichnser_seifert2018} and 
\cite[Remark 2.3 (c), p.~3]{kruse_schwenninger2022} 
we may choose $\mathcal{P}_{\gamma}$ such that $\|x\|=\sup_{p_{\gamma}\in\mathcal{P}_{\gamma}}p_{\gamma}(x)$ 
for all $x\in X$ and hence we get
\[
 \|x\|^{\bullet}\geq M^{-1}\sup_{p_{\gamma}\in\mathcal{P}_{\gamma}}p_{\gamma}(x)=M^{-1}\|x\|.
\]
\end{proof}

The proof shows that we actually have $\|x\|^{\bullet}\leq \|x\|\leq M\|x\|^{\bullet}$ for all $x\in X$ with 
$M\coloneqq\limsup_{t\to 0\rlim}\|T(t)\|_{\mathcal{L}(X)}$. 

\begin{rem}\label{rem:bi_sun_sun}
Let $(X,\|\cdot\|,\tau)$ be a sequentially complete d-consistent Saks space,  
$(T(t))_{t\geq 0}$ a bi-continuous semigroup on $X$ with generator $(A,D(A))$ and 
set ${X^{\bullet}}'\coloneqq(X^{\bullet})'$, ${T^{\bullet}}'(t)\coloneqq(T^{\bullet})'(t)$ for $t\geq 0$ and 
${A^{\bullet}}'\coloneqq(A^{\bullet})'$. 
Then $({T^{\bullet}}'(t))_{t\geq 0}$ is a bi-continuous semigroup 
on $({X^{\bullet}}',\|\cdot\|_{{X^{\bullet}}'},\sigma({X^{\bullet}}',X^{\bullet}))$ by \cite[Proposition 3.18, p.~78]{kuehnemund2001} and \prettyref{cor:bi_sun_dual} and we define 
the \emph{bi-sun-sun dual}
\[
X^{\bullet\bullet}\coloneqq\{{x^{\bullet}}'\in {X^{\bullet}}'\;|\;
\lim_{t\to 0\rlim}\|{T^{\bullet}}'(t){x^{\bullet}}'-{x^{\bullet}}'\|_{{X^{\bullet}}'}=0\}.
\]
Then $X^{\bullet\bullet}$ is a $\|\cdot\|_{{X^{\bullet}}'}$-closed, sequentially 
$\tau_{\operatorname{c}}({X^{\bullet}}',X^{\bullet})$-dense, 
$({T^{\bullet}}'(t))_{t\geq 0}$-invariant subspace 
of ${X^{\bullet}}'$. Moreover, $X^{\bullet\bullet}=\overline{D({A^{\bullet}}')}^{\|\cdot\|_{{X^{\bullet}}'}}$ and 
$(T^{\bullet\bullet}(t))_{t\geq 0}\coloneqq({T^{\bullet}}'(t)_{\mid X^{\bullet\bullet}})_{t\geq 0}$ is the 
$\|\cdot\|_{{X^{\bullet}}'}$-strongly continuous semigroup on $X^{\bullet\bullet}$ generated by the part 
$A^{\bullet\bullet}$ of ${A^{\bullet}}'$ in $X^{\bullet\bullet}$ 
by \prettyref{thm:generator} (g), \prettyref{cor:bi_sun_dual} and since
\[
 \gamma(\|\cdot\|_{{X^{\bullet}}'},\sigma({X^{\bullet}}',X^{\bullet}))
=\tau_{\operatorname{c}}({X^{\bullet}}',X^{\bullet})
\]
by \cite[Example E), p.~66]{wiweger1961}.
If $\tau$ coincides with the $\|\cdot\|$-topology, then $X^{\bullet\bullet}=X^{\odot\odot}$, which is the 
\emph{sun-sun dual} (see \cite[p.~7]{vanneerven1992}).
\end{rem}

Let us comment on the definition of $X^{\bullet\bullet}$ and its relation to 
$X^{\bullet\circ}\coloneqq(X^{\bullet})^{\circ}$ and $(X^{\bullet})^{\bullet}$.

\begin{rem}
Let $(X,\|\cdot\|,\tau)$ be a sequentially complete d-consistent Saks space  
and $(T(t))_{t\geq 0}$ a bi-continuous semigroup on $X$. 
The triple $(X^{\bullet},\|\cdot\|_{X^{\bullet}},\sigma(X^{\bullet},{X^{\bullet}}'))$ is a Saks space, 
where $\|\cdot\|_{X^{\bullet}}$ denotes the restriction of $\|\cdot\|_{X'}$ to $X^{\bullet}$,
and we have 
\[
X^{\bullet\circ}=\{{x^{\bullet}}'\in {X^{\bullet}}'\;|\;{x^{\bullet}}'\;\sigma(X^{\bullet},{X^{\bullet}}')\text{-sequentially continuous on } \|\cdot\|_{X^{\bullet}}\text{-bounded sets}\}.
\]
For ${x^{\bullet}}'\in {X^{\bullet}}'$ we note that 
\[
|\langle{x^{\bullet}}',x^{\bullet}\rangle|=\sup_{y\in\{{x^{\bullet}}'\}}|\langle y,x^{\bullet}\rangle|
=:p_{\{{x^{\bullet}}'\}}(x^{\bullet})
\]
for all $x^{\bullet}\in X^{\bullet}$, which means that ${x^{\bullet}}'$ is $\sigma(X^{\bullet},{X^{\bullet}}')$-continuous. 
This implies $X^{\bullet\circ}={X^{\bullet}}'$. Setting 
\[
T^{\bullet\circ}(t)\coloneqq{T^{\bullet}}'(t)_{\mid X^{\bullet\circ}}={T^{\bullet}}'(t)_{\mid {X^{\bullet}}'}={T^{\bullet}}'(t)
\]
for all $t\geq 0$, we see that
\[
(X^{\bullet})^{\bullet}=\{x^{\bullet\circ}\in X^{\bullet\circ}\;|\;\lim_{t\to 0\rlim}\|T^{\bullet\circ}(t)x^{\bullet\circ}-x^{\bullet\circ}\|_{X^{\bullet\circ}}=0\}=X^{\bullet\bullet}.
\]
\end{rem}

Like in \cite[Corollary 1.3.6, p.~8]{vanneerven1992} we can consider $X$ as a subspace of ${X^{\bullet}}'$.

\begin{cor}\label{cor:embed_bi_sun_dual}
Let $(X,\|\cdot\|,\tau)$ be a sequentially complete d-consistent Mazur--Saks space 
and $(T(t))_{t\geq 0}$ a bi-continuous semigroup on $X$. 
Then the canonical map $j\colon X\to {X^{\bullet}}'$ given by
\[
\langle j(x), x^{\bullet}\rangle\coloneqq\langle x^{\bullet},x\rangle,\quad x\in X,\,x^{\bullet}\in X^{\bullet},
\]
is injective, $j(X_{\operatorname{cont}})= X^{\bullet\bullet}\cap j(X)$ and
$j\in\mathcal{L}(X;{X^{\bullet}}')
\coloneqq\mathcal{L}((X,\|\cdot\|);({X^{\bullet}}',\|\cdot\|_{{X^{\bullet}}'}))$ 
with $M^{-1}\leq\|j\|_{\mathcal{L}(X;{X^{\bullet}}')}\leq 1$ where 
$M\coloneqq\limsup_{t\to 0\rlim}\|T(t)\|_{\mathcal{L}(X)}$. 
\end{cor}
\begin{proof}
$j$ is clearly linear. If $j(x)=0$ for some $x\in X$, 
then $\langle x^{\bullet},x\rangle=0$ for all $x^{\bullet}\in X^{\bullet}$, 
which implies that $\|x\|^{\bullet}=0$ and thus $x=0$ by \prettyref{thm:eqiv_norms}. 

The inclusion $j(X_{\operatorname{cont}})\subset (X^{\bullet\bullet}\cap j(X))$ follows 
directly from the definitions of $X_{\operatorname{cont}}$ (see \prettyref{thm:generator} (g)) and 
$X^{\bullet\bullet}$. For the converse inclusion let $x\in X$ with $j(x)\in X^{\bullet\bullet}$. 
We note that for any $t\geq 0$ and $x^{\bullet}\in X^{\bullet}$
\begin{align*}
 \langle T^{\bullet\bullet}(t)j(x)-j(x),x^{\bullet}\rangle 
&=\langle {T^{\bullet}}'(t)j(x)-j(x),x^{\bullet}\rangle 
 =\langle j(x),T^{\bullet}(t)x^{\bullet}-x^{\bullet}\rangle \\
&=\langle T'(t)x^{\bullet}-x^{\bullet},x\rangle 
 =\langle x^{\bullet},T(t)x-x\rangle ,
\end{align*}
which implies $\|T(t)x-x\|^{\bullet}=\|T^{\bullet\bullet}(t)j(x)-j(x)\|_{{X^{\bullet}}'}$ and thus 
$x\in X_{\operatorname{cont}}$ as $\|\cdot\|^{\bullet}$ and $\|\cdot\|$ are equivalent 
by \prettyref{thm:eqiv_norms}.

Furthermore, we have 
\[
 \|j\|_{\mathcal{L}(X;{X^{\bullet}}')}
=\sup_{\substack{x\in X\\ \|x\|\leq 1}}\|j(x)\|_{{X^{\bullet}}'}
=\sup_{\substack{x\in X\\ \|x\|\leq 1}}
  \sup_{\substack{x^{\bullet}\in X^{\bullet}\\ \|x^{\bullet}\|_{X'}\leq 1}}
  |\langle x^{\bullet},x\rangle|
=\sup_{\substack{x\in X\\ \|x\|\leq 1}}\|x\|^{\bullet},
\]
implying the rest of our statement because $\|x\|^{\bullet}\leq \|x\|\leq M\|x\|^{\bullet}$ for all $x\in X$ 
with $M\coloneqq\limsup_{t\to 0\rlim}\|T(t)\|_{\mathcal{L}(X)}$.
\end{proof}

In our next theorem we investigate the relation between the resolvent sets $\rho(A)$, $\rho(A^{\bullet})$ 
and $\rho({A^{\bullet}}')$
resp.~the resolvents $R(\lambda,A)$, $R(\lambda,A^{\bullet})$ and $R(\lambda,{A^{\bullet}}')$.

\begin{thm}\label{thm:resolv_bullet}
Let $(X,\|\cdot\|,\tau)$ be a sequentially complete d-consistent Saks space 
and $(T(t))_{t\geq 0}$ a bi-continuous semigroup on $X$ with generator $(A,D(A))$. 
For $\lambda\in\rho(A)$ we set $R(\lambda,A)^{\bullet}\coloneqq R(\lambda,A)'_{\mid X^{\bullet}}$.
\begin{enumerate}
\item[(a)] If $\lambda\in\rho(A)$ such that $R(\lambda,A)^{\bullet}X^{\bullet}\subset D(A^{\bullet})$, 
then we have $\lambda\in \rho(A^{\bullet})$ and $R(\lambda,A)^{\bullet}=R(\lambda,A^{\bullet})$.
\item[(b)] If $(X,\gamma)$ is a Mazur space and $\lambda\in\rho(A^{\bullet})$, then we have $\lambda\in \rho(A)$.
\item[(c)] We have $\rho(A^{\bullet})=\rho({A^{\bullet}}')$ and 
$R(\lambda,A^{\bullet})'=R(\lambda,{A^{\bullet}}')$ for all $\lambda\in\rho(A^{\bullet})$. 
If $\lambda\in\rho(A)$ such that $R(\lambda,A)^{\bullet}X^{\bullet}\subset D(A^{\bullet})$, 
then we have $j(R(\lambda,A)x)=R(\lambda,{A^{\bullet}}')j(x)$ for all $X$ with the canonical map 
$j\colon X\to {X^{\bullet}}'$.
\end{enumerate}
\end{thm}
\begin{proof}
(a) Let $\lambda\in\rho(A)$. For any $x\in X$ and $x^{\bullet}\in D(A^{\bullet})$ we have 
$A^{\bullet}x^{\bullet}=A^{\circ}x^{\bullet}\in X^{\bullet}$ since 
$A^{\bullet}$ is the part of $A^{\circ}$ in $X^{\bullet}$ by \prettyref{cor:bi_sun_dual}, and 
\[
 \langle R(\lambda,A)^{\bullet}(\lambda-A^{\bullet})x^{\bullet},x\rangle
=\langle R(\lambda,A)'(\lambda-A^{\circ})x^{\bullet},x\rangle
=\langle x^{\bullet},(\lambda-A)R(\lambda,A)x\rangle
=\langle x^{\bullet},x\rangle,
\] 
which implies $R(\lambda,A)^{\bullet}(\lambda-A^{\bullet})x^{\bullet}=x^{\bullet}$. 
From the assumption $R(\lambda,A)^{\bullet}X^{\bullet}\subset D(A^{\bullet})$ and \prettyref{cor:bi_sun_dual} 
we deduce that $(\lambda-A^{\bullet})R(\lambda,A)^{\bullet}x^{\bullet}\in X^{\bullet}$ and 
for any $x\in D(A)$ we have 
\[
 \langle (\lambda-A^{\bullet})R(\lambda,A)^{\bullet}x^{\bullet},x\rangle
=\langle (\lambda-A^{\circ})R(\lambda,A)'x^{\bullet},x\rangle
=\langle x^{\bullet},R(\lambda,A)(\lambda-A)x\rangle
=\langle x^{\bullet},x\rangle.
\] 
As  $(\lambda-A^{\bullet})R(\lambda,A)^{\bullet}x^{\bullet},x^{\bullet}\in X^{\bullet}\subset X^{\circ}$ 
and $D(A)$ is sequentially $\gamma$-dense by \prettyref{thm:generator} (b), 
we get $(\lambda-A^{\bullet})R(\lambda,A)^{\bullet}x^{\bullet}=x^{\bullet}$ from the definition of $X^{\circ}$. 
Hence we obtain $\lambda\in\rho(A^{\bullet})$ and $R(\lambda,A)^{\bullet}=R(\lambda,A^{\bullet})$.

(b) Conversely, let $\lambda\in\rho(A^{\bullet})$. If $(\lambda-A)x=0$ for some $x\in D(A)$, then for all 
$x^{\circ}\in D(A^{\circ})$ we have 
\[
 \langle (\lambda-A^{\circ})x^{\circ},x\rangle
=\langle x^{\circ},(\lambda-A)x\rangle=0,
\] 
which means that $x$ annihilates the range of $\lambda-A^{\circ}$. 
In particular, $x$ annihilates $(\lambda-A^{\bullet})D(A^{\bullet})=X^{\bullet}$ by \prettyref{cor:bi_sun_dual} 
because $\lambda\in\rho(A^{\bullet})$. Thus we have $\|x\|^{\bullet}=0$ and so $x=0$ 
by \prettyref{thm:eqiv_norms}, implying the injectivity of $\lambda-A$. 

Next, we show that the range of $\lambda-A$ is $\|\cdot\|$-dense and $\|\cdot\|$-closed, which 
then implies the surjectivity of $\lambda-A$.
Suppose that the range of $\lambda-A$ is not $\|\cdot\|$-dense. 
Then there is some $x^{\circ}\in X^{\circ}$ with $x^{\circ}\neq 0$ such that for any $x\in D(A)$ we have
\[
\langle (\lambda-A)x, x^{\circ}\rangle =0
\]
since $ X^{\circ}$ separates the points of $X$.
It follows that $\langle Ax, x^{\circ}\rangle=\langle x,\lambda x^{\circ}\rangle$ for all $x\in D(A)$ 
and so $x^{\circ}\in D(A^{\circ})$ by \prettyref{prop:adj_bi_cont} (b). We deduce that 
$(\lambda-A^{\circ})x^{\circ}=0$ as $D(A)$ is sequentially $\gamma$-dense by \prettyref{thm:generator} (b), which yields 
$A^{\circ}x^{\circ}=\lambda x^{\circ}\in D(A^{\circ})\subset X^{\bullet}$. 
Thus $x^{\circ}\in D(A^{\bullet})$ because $A^{\bullet}$ is the part of $A^{\circ}$ in $X^{\bullet}$ 
by \prettyref{cor:bi_sun_dual}. We conclude that 
\[
(\lambda-A^{\bullet})x^{\circ}=(\lambda-A^{\circ})x^{\circ}=0
\]
with $x^{\circ}\neq 0$, which contradicts $\lambda\in\rho(A^{\bullet})$. 

Let us turn to the $\|\cdot\|$-closedness of the range of $\lambda-A$. Let $x\in D(A)$. 
By \prettyref{thm:eqiv_norms} there is $x^{\bullet}\in X^{\bullet}$ with $\|x^{\bullet}\|_{X'}\leq 1$ such that 
$|\langle x^{\bullet},x \rangle|\geq \tfrac{1}{2} \|x\|^{\bullet}$ due to $(X,\gamma)$ being a Mazur space.
Setting $C\coloneqq\|R(\lambda,A^{\bullet})\|_{\mathcal{L}(X^{\bullet})}^{-1}$, we note that 
\begin{align}\label{eq:resolv_bullet_0}
     \|(\lambda-A)x\|^{\bullet}
&\geq C|\langle R(\lambda,A^{\bullet})x^{\bullet},(\lambda-A)x \rangle|
 = C|\langle (\lambda-A^{\bullet})R(\lambda,A^{\bullet})x^{\bullet},x \rangle|
 = C|\langle x^{\bullet},x \rangle|\nonumber\\
&\geq \frac{C}{2} \|x\|^{\bullet}.
\end{align}
Now, if $(x_{n})_{n\in\N}$ is a sequence in $D(A)$ such that 
$\|\cdot\|$-$\lim_{n\to\infty}(\lambda-A)x_{n}=y$ for some $y\in X$, 
we derive from the estimate above that $(x_{n})_{n\in\N}$ is a $\|\cdot\|$-Cauchy sequence, 
say with limit $z\in X$, because $\|\cdot\|$ and $\|\cdot\|^{\bullet}$ are equivalent norms on $X$ 
by \prettyref{thm:eqiv_norms}. Since $(A,D(A))$ is sequentially $\gamma$-closed by \prettyref{thm:generator} (a), in particular 
$\|\cdot\|$-closed as $\gamma$ is coarser than the $\|\cdot\|$-topology, we get $z\in D(A)$ and 
$y=(\lambda-A)z$. Hence $\lambda-A$ is bijective and \eqref{eq:resolv_bullet_0} yields that 
$(\lambda-A)^{-1}\in\mathcal{L}(X)$ as well. 

(c)  By \cite[Lemma 1.4.1, p.~9]{vanneerven1992} and \prettyref{cor:bi_sun_dual} it holds 
$\rho(A^{\bullet})=\rho({A^{\bullet}}')$ and $R(\lambda,A^{\bullet})'=R(\lambda,{A^{\bullet}}')$ for all 
$\lambda\in\rho(A^{\bullet})$. 
Now, let $\lambda\in\rho(A)$ such that $R(\lambda,A)^{\bullet}X^{\bullet}\subset D(A^{\bullet})$.
Then it follows from part (a) that $\lambda\in\rho(A^{\bullet})$ and 
$R(\lambda,A)^{\bullet}=R(\lambda,A^{\bullet})$. Thus we have 
\begin{align*}
 \langle j(R(\lambda,A)x),x^{\bullet}\rangle
&=\langle x^{\bullet},R(\lambda,A)x\rangle
 =\langle R(\lambda,A)^{\bullet}x^{\bullet},x\rangle
 =\langle R(\lambda,A^{\bullet})x^{\bullet},x\rangle\\
&=\langle j(x),R(\lambda,A^{\bullet})x^{\bullet}\rangle
 =\langle R(\lambda,A^{\bullet})'j(x),x^{\bullet}\rangle
 =\langle R(\lambda,{A^{\bullet}}')j(x),x^{\bullet}\rangle
\end{align*}
for all $x\in X$ and $x^{\bullet}\in X^{\bullet}$, meaning
$j(R(\lambda,A)x)=R(\lambda,{A^{\bullet}}')j(x)$ for all $x\in X$.
\end{proof}

Let us turn to sufficient conditions for $R(\lambda,A)^{\bullet}X^{\bullet}\subset D(A^{\bullet})$ 
to hold in \prettyref{thm:resolv_bullet} (a). 

\begin{prop}\label{prop:resolv_bullet}
Let $(X,\|\cdot\|,\tau)$ be a sequentially complete d-consistent Saks space 
and $(T(t))_{t\geq 0}$ a bi-continuous semigroup on $X$ with generator $(A,D(A))$. 
\begin{enumerate}
\item[(a)] If $\re\lambda>\omega_{0}(T)$, then we have $\lambda\in\rho(A)$ and 
$R(\lambda,A)^{\bullet}X^{\bullet}\subset D(A^{\bullet})$.
\item[(b)] If $(X,\gamma)$ is a Mazur space and $\lambda\in\rho(A)$ such that 
$R(\lambda,A)\colon (X,\gamma)\to(X,\gamma)$ is continuous, 
then we have $R(\lambda,A)^{\bullet}X^{\bullet}\subset D(A^{\bullet})$.
\item[(c)] If $(X,\gamma)$ is a C-sequential space, then 
$\{R(\lambda,A)\;|\;\re\lambda\geq\alpha\}$ is $\gamma$-equiconti\-nuous 
for all $\alpha>\omega_{0}(T)$. 
Especially, $R(\lambda,A)\colon (X,\gamma)\to(X,\gamma)$ is continuous for all $\re\lambda>\omega_{0}(T)$.
\end{enumerate}
\end{prop}
\begin{proof}
(a) Let $\re\lambda>\omega_{0}(T)$. Then we have $\re\lambda>\omega_{0}(T^{\bullet})$ by \prettyref{cor:bi_sun_dual} 
and thus $\lambda\in\rho(A)\cap\rho(A^{\bullet})$ by \prettyref{thm:generator} (e) as well as 
\begin{align*}
  \langle R(\lambda,A)^{\bullet}x^{\bullet},x\rangle
&=\langle x^{\bullet},R(\lambda,A)x\rangle
 =\langle x^{\bullet},\int_{0}^{\infty}\e^{-\lambda s}T(s)x\d s\rangle\\
&=\int_{0}^{\infty}\e^{-\lambda s}\langle x^{\bullet},T(s)x\rangle\d s
 =\int_{0}^{\infty}\e^{-\lambda s}\langle T^{\bullet}(s)x^{\bullet},x\rangle\d s\\
&=\langle\int_{0}^{\infty}\e^{-\lambda s} T^{\bullet}(s)x^{\bullet}\d s,x\rangle
 =\langle R(\lambda,A^{\bullet})x^{\bullet},x\rangle
\end{align*}
for all $x^{\bullet}\in X^{\bullet}$ and $x\in X$. 
This yields $R(\lambda,A)^{\bullet}x^{\bullet}= R(\lambda,A^{\bullet})x^{\bullet}\in D(A^{\bullet})$ for all 
$x^{\bullet}\in X^{\bullet}$.

(b) Let $\lambda\in\rho(A)$ such that $R(\lambda,A)\colon (X,\gamma)\to(X,\gamma)$ is continuous. 
First, we show that $R(\lambda,A)^{\bullet}D(A^{\circ})\subset D(A^{\bullet})$. 
Let $x^{\circ}\in D(A^{\circ})\subset X^{\bullet}\subset X^{\circ}$ and $\mathcal{P}_{\gamma}$ be a directed system 
of seminorms that generates the mixed topology $\gamma$. 
Since $X^{\circ}=(X,\gamma)'$ by \prettyref{rem:bi_cont_dual}, 
there are $p_{\gamma}\in\mathcal{P}_{\gamma}$ and $C\geq 0$ such that 
\[
  |\langle R(\lambda,A)^{\bullet}x^{\circ},x\rangle|
= |\langle x^{\circ},R(\lambda,A)x\rangle|
\leq Cp_{\gamma}(R(\lambda,A)x)
\]
for all $x\in X$. Due to the continuity of $R(\lambda,A)\colon (X,\gamma)\to(X,\gamma)$, there are 
$\widetilde{p}_{\gamma}\in\mathcal{P}_{\gamma}$ and $\widetilde{C}\geq 0$ such that 
\[
  |\langle R(\lambda,A)^{\bullet}x^{\circ},x\rangle|
\leq C\widetilde{C}\widetilde{p}_{\gamma}(x)
\]
for all $x\in X$, implying $R(\lambda,A)^{\bullet}x^{\circ}\in(X,\gamma)'=X^{\circ}$. 
For any $x\in D(A)$ we have 
\[
 \langle Ax,R(\lambda,A)^{\bullet}x^{\circ}\rangle
=\langle R(\lambda,A)Ax,x^{\circ}\rangle
=\langle \lambda R(\lambda,A)x-x,x^{\circ}\rangle.
\]
For any $x\in X$ we note that 
\[
|y^{\circ}(x)|
\coloneqq|\langle \lambda R(\lambda,A)x-x,x^{\circ}\rangle|
\leq Cp_{\gamma}(\lambda R(\lambda,A)x-x)
\leq C\widetilde{C}|\lambda|\widetilde{p}_{\gamma}(x)+Cp_{\gamma}(x),
\]
which means that $y^{\circ}\in (X,\gamma)'=X^{\circ}$. Thus we have 
\[
 \langle Ax,R(\lambda,A)^{\bullet}x^{\circ}\rangle=\langle x,y^{\circ}\rangle
\] 
for all $x\in D(A)$, i.e.~$R(\lambda,A)^{\bullet}x^{\circ}\in D(A^{\circ})$ by \prettyref{prop:adj_bi_cont} (b).
Further, we observe that for any $x\in D(A)$
\[
 \langle A^{\circ}R(\lambda,A)^{\bullet}x^{\circ},x\rangle
=\langle x^{\circ},R(\lambda,A)Ax\rangle
=\langle x^{\circ},\lambda R(\lambda,A)x-x\rangle
=\langle \lambda R(\lambda,A)^{\bullet}x^{\circ}-x^{\circ},x\rangle.
\]
The sequential $\gamma$-density of $D(A)$ by \prettyref{thm:generator} (b) 
and $A^{\circ}R(\lambda,A)^{\bullet}x^{\circ}\in X^{\circ}$ as well as
$\lambda R(\lambda,A)^{\bullet}x^{\circ}-x^{\circ}=y^{\circ}\in X^{\circ}$ imply that for any $x\in X$
\begin{equation}\label{eq:resolv_bullet_1}
 \langle A^{\circ}R(\lambda,A)^{\bullet}x^{\circ},x\rangle
=\langle \lambda R(\lambda,A)^{\bullet}x^{\circ}-x^{\circ},x\rangle.
\end{equation}
Thus we have for any $x\in D(A)$  
\begin{align*}
  \langle Ax,A^{\circ}R(\lambda,A)^{\bullet}x^{\circ}\rangle
&=\langle Ax,\lambda R(\lambda,A)^{\bullet}x^{\circ}-x^{\circ}\rangle
 =\langle\lambda R(\lambda,A)Ax, x^{\circ}\rangle-\langle x,A^{\circ}x^{\circ}\rangle\\
&=\langle\lambda^{2} R(\lambda,A)x-\lambda x, x^{\circ}\rangle-\langle x,A^{\circ}x^{\circ}\rangle.
\end{align*}
For any $x\in X$ we remark that
\[
  |z^{\circ}(x)|
\coloneqq|\langle\lambda^{2} R(\lambda,A)x-\lambda x, x^{\circ}\rangle|
\leq Cp_{\gamma}(\lambda^{2} R(\lambda,A)x-\lambda x)
\leq C\widetilde{C}|\lambda|^{2}\widetilde{p}_{\gamma}(x)+C|\lambda|p_{\gamma}(x),
\]
yielding $z^{\circ}\in (X,\gamma)'=X^{\circ}$. It follows that $z^{\circ}-A^{\circ}x^{\circ}\in X^{\circ}$ and 
\[
 \langle Ax,A^{\circ}R(\lambda,A)^{\bullet}x^{\circ}\rangle
=\langle x, z^{\circ}-A^{\circ}x^{\circ}\rangle
\]
for all $x\in D(A)$, i.e.~$A^{\circ}R(\lambda,A)^{\bullet}x^{\circ}\in D(A^{\circ})$ 
by \prettyref{prop:adj_bi_cont} (b). 
We conclude that $R(\lambda,A)^{\bullet}x^{\circ}\in D(A^{\bullet})$ since $A^{\bullet}$ is the part of $A^{\circ}$ 
in $X^{\bullet}$ by \prettyref{cor:bi_sun_dual}. Thus we have shown that 
$R(\lambda,A)^{\bullet}D(A^{\circ})\subset D(A^{\bullet})$. Now, we show that 
$R(\lambda,A)^{\bullet}X^{\bullet}\subset D(A^{\bullet})$. First, we observe that for any $x^{\bullet}\in X^{\bullet}$
\[
\|R(\lambda,A)^{\bullet}x^{\bullet}\|_{X'}
=\sup_{\substack{x\in X\\ \|x\|\leq 1}}|\langle x^{\bullet},R(\lambda,A)x\rangle|
\leq \|R(\lambda,A)\|_{\mathcal{L}(X)}\|x^{\bullet}\|_{X'},
\]
implying that $R(\lambda,A)^{\bullet}\in\mathcal{L}(X^{\bullet};X')$. 
Since $R(\lambda,A)^{\bullet}x^{\circ}\in D(A^{\bullet})$ for any $x^{\circ}\in D(A^{\circ})$, we have 
\[
 \langle A^{\bullet}R(\lambda,A)^{\bullet}x^{\circ},x\rangle
=\langle A^{\circ}R(\lambda,A)^{\bullet}x^{\circ},x\rangle
\underset{\eqref{eq:resolv_bullet_1}}{=}\langle \lambda R(\lambda,A)^{\bullet}x^{\circ}-x^{\circ},x\rangle
=\langle x^{\circ},\lambda R(\lambda,A)x-x\rangle
\]
for all $x\in X$. We deduce that 
\begin{equation}\label{eq:resolv_bullet_2}
     \|A^{\bullet}R(\lambda,A)^{\bullet}x^{\circ}\|_{X'}
\leq |\lambda|\|R(\lambda,A)\|_{\mathcal{L}(X)}\|x^{\circ}\|_{X'}+\|x^{\circ}\|_{X'}
\end{equation}
for all $x^{\circ}\in D(A^{\circ})$. Let $x^{\bullet}\in X^{\bullet}$. Due to \prettyref{cor:bi_sun_dual} it holds
$X^{\bullet}=\overline{D(A^{\circ})}^{\|\cdot\|_{X^{\circ}}}=\overline{D(A^{\circ})}^{\|\cdot\|_{X'}}$ 
and thus there is a sequence $(x_{n}^{\circ})_{n\in\N}$ 
in $D(A^{\circ})$ which $\|\cdot\|_{X'}$-converges to $x^{\bullet}$. 
From $R(\lambda,A)^{\bullet}\in\mathcal{L}(X^{\bullet};X')$ we derive that the sequence
$(R(\lambda,A)^{\bullet}x_{n}^{\circ})_{n\in\N}$ in $D(A^{\bullet})$ $\|\cdot\|_{X'}$-converges 
to $R(\lambda,A)^{\bullet}x^{\bullet}\in \overline{D(A^{\circ})}^{\|\cdot\|_{X'}}=X^{\bullet}$. The estimate \eqref{eq:resolv_bullet_2} implies that 
$(A^{\bullet}R(\lambda,A)^{\bullet}x_{n}^{\circ})_{n\in\N}$ is a $\|\cdot\|_{X'}$-Cauchy sequence 
in $X^{\bullet}$. The space $X^{\bullet}$ is $\|\cdot\|_{X'}$-complete by \prettyref{cor:bi_sun_dual}, 
which yields that $(A^{\bullet}R(\lambda,A)^{\bullet}x_{n}^{\circ})_{n\in\N}$ $\|\cdot\|_{X'}$-converges to some 
$w^{\bullet}\in X^{\bullet}$. In combination with the $\|\cdot\|_{X'}$-closedness of the generator 
$(A^{\bullet},D(A^{\bullet}))$ by \prettyref{cor:bi_sun_dual} we get 
$R(\lambda,A)^{\bullet}x^{\bullet}\in D(A^{\bullet})$ and $w^{\bullet}=A^{\bullet}R(\lambda,A)^{\bullet}x^{\bullet}$.

(c) If $(X,\gamma)$ is C-sequential, then \cite[Condition C, p.~165--166]{kraaij2016} is fulfilled 
by \cite[Proposition 7.3, p.~179]{kraaij2016} and \cite[Theorem 7.4, p.~52]{wilansky1981}. 
Further, $(T(t))_{t\geq 0}$ is an SCLE-semigroup w.r.t.~$\gamma$ in the sense of \cite[p.~160]{kraaij2016}, 
i.e.~strongly continuous w.r.t.~$\gamma$ and locally equicontinuous w.r.t.~$\gamma$, 
by \cite[Theorem 7.4, p.~180]{kraaij2016}. 
Therefore $\{R(\lambda,A)\;|\;\re\lambda\geq\alpha\}$ is $\gamma$-equicontinuous 
for all $\alpha>\omega_{0}(T)$ by \cite[Theorem 6.4 (a)$\Leftrightarrow$(c), p.~176]{kraaij2016}. 
\end{proof}

Part (a) shows that the continuity of $R(\lambda,A)\colon (X,\gamma)\to(X,\gamma)$ need not be a necessary 
condition for $R(\lambda,A)^{\bullet}X^{\bullet}\subset D(A^{\bullet})$ for all $\re\lambda>\omega_{0}(T)$. 
This is an open question. Another open question is whether one actually has 
$R(\lambda,A)^{\bullet}X^{\bullet}\subset D(A^{\bullet})$ for all $\lambda\in\rho(A)$ in general. 
The answer is affirmative if $\tau$ coincides with the $\|\cdot\|$-topology. Because 
then $\gamma$ also coincides with the $\|\cdot\|$-topology, 
which gives that $R(\lambda,A)\colon (X,\gamma)\to(X,\gamma)$ is continuous for all $\lambda\in\rho(A)$. 
Therefore \prettyref{prop:resolv_bullet} (b) and \prettyref{thm:resolv_bullet} imply 
\cite[Theorem 1.4.2, p.~10]{vanneerven1992} (see also \cite[Theorem 14.3.3, p.~425]{hille1996}).

Let us come to an application of \prettyref{prop:resolv_bullet} (b) where we do not 
need the restriction that $\re\lambda>\omega_{0}(T)$ or that $\tau$ coincides with the $\|\cdot\|$-topology. 
We note that $\mathrm{C}_{\operatorname{b}}(\N)=\ell^{\infty}$ and $\mathrm{M}_{\operatorname{t}}(\N)=\ell^{1}$ 
(see e.g.~\cite[p.~477]{conway1967}), implying 
$\beta_{0}=\gamma(\|\cdot\|_{\infty},\tau_{\operatorname{co}})=\mu(\ell^{\infty},\ell^{1})$ 
by \prettyref{ex:mackey_mazur_examples} (b).
Further, it follows from \cite[p.~22]{kruse_schwenninger2022} (or \prettyref{ex:mackey_mazur_examples} (e)) 
that the triple $(\ell^{\infty},\|\cdot\|_{\infty},\mu(\ell^{\infty},\ell^{1}))$ is 
a sequentially complete Saks space and hence from \prettyref{defn:mixed_top_Saks} (a) 
that a bi-continuous semigroup on $(\ell^{\infty},\|\cdot\|_{\infty},\tau_{\operatorname{co}})$
is a bi-continuous semigroup on $(\ell^{\infty},\|\cdot\|_{\infty},\mu(\ell^{\infty},\ell^{1}))$ as well. 

\begin{exa}\label{ex:resolv_bullet} 
Let $q\colon \N\to\C$ be such that $\sup_{n\in\N}\re q(n)<\infty$, and let $(T(t))_{t\geq 0}$ be the 
bi-continuous multiplication semigroup on $(\ell^{\infty},\|\cdot\|_{\infty},\mu(\ell^{\infty},\ell^{1}))$ given by 
\[
T(t)x\coloneqq(\e^{q(n)t}x_{n})_{n\in\N},\quad x\in\ell^{\infty},\, t\geq 0.
\]
Then the generator $(A,D(A))$ of $(T(t))_{t\geq 0}$ is the multiplication operator 
$
A\colon D(A)\to \ell^{\infty},\;Ax =qx,
$
with domain $D(A)=\{x\in\ell^{\infty}\;|\;(q(n)x_{n})_{n\in\N}\in\ell^{\infty}\}$ by \cite[p.~353--354]{budde2019}.
Furthermore, we have $\sigma(A)\coloneqq\C\setminus\rho(A)=\overline{q(\N)}$ by \cite[Chap.~I, 4.8 Exercises (1), p.~30]{engel_nagel2000} 
and 
\[
 R(\lambda,A)x
=(\lambda-A)^{-1}x
=\Bigl(\frac{1}{\lambda-q(n)}x_{n}\Bigr)_{n\in\N},\quad x\in\ell^{\infty},\,\lambda\notin \sigma(A).
\]
Next, we show that $R(\lambda,A)$ is $\mu(\ell^{\infty},\ell^{1})$-continuous for all $\lambda\in\rho(A)$.
Due to \cite[Theorem III.2.15, p.~76]{diestel1977} a set $M\subset\ell^{1}$ is 
$\sigma(\ell^{1},\ell^{\infty})$-compact if and only if $M$ is $\|\cdot\|_{\ell^{1}}$-bounded and 
uniformly absolutely summable, i.e.
\[
\forall\;\varepsilon>0\;\exists\;\delta>0\;\forall\;\Omega\subset\N,\,|\Omega|<\delta,\,y\in M:\;
\sum_{n\in\Omega}|y_{n}|<\varepsilon,
\]
where $|\Omega|$ denotes the cardinality of $\Omega$.
Let $M\subset\ell^{1}$ be $\sigma(\ell^{1},\ell^{\infty})$-compact and absolutely convex. 
Then we have 
\begin{equation}\label{eq:ex_resolv_bullet}
 \sup_{y\in M}|\langle R(\lambda,A)x,y\rangle|
=\sup_{y\in M}\Bigl|\sum_{n\in\N}\frac{1}{\lambda-q(n)}x_{n}y_{n}\Bigr|
=\sup_{y\in M_{\lambda}}|\langle x, y\rangle |
\end{equation}
for all $x\in\ell^{\infty}$ and $\lambda\notin \sigma(A)$ where 
\[
M_{\lambda}\coloneqq\Bigl\{\Bigl(\frac{1}{\lambda-q(n)}y_{n}\Bigr)_{n\in\N}\;|\;y\in M\Bigr\}.
\]
Now, we only need to show that $M_{\lambda}$ is $\sigma(\ell^{1},\ell^{\infty})$-compact and absolutely convex. 
First, we note that $C_{\lambda}\coloneqq\sup_{n\in\N}\tfrac{1}{|\lambda-q(n)|}<\infty$ for all 
$\lambda\notin \sigma(A)=\overline{q(\N)}$ and 
\[
 \Bigr\|\Bigl(\frac{1}{\lambda-q(n)}y_{n}\Bigr)_{n\in\N}\Bigl\|_{\ell^{1}}
=\sum_{n\in\N}\frac{1}{|\lambda-q(n)|}|y_{n}|
\leq \sup_{n\in\N}\frac{1}{|\lambda-q(n)|}\|y\|_{\ell^{1}}
=C_{\lambda}\|y\|_{\ell^{1}}
\]
for all $y\in M$, which implies that $M_{\lambda}$ is $\|\cdot\|_{\ell^{1}}$-bounded because 
$M$ is $\|\cdot\|_{\ell^{1}}$-bounded. 
Due to the characterisation of $\sigma(\ell^{1},\ell^{\infty})$-compactness above 
it remains to show that $M_{\lambda}$ is uniformly absolutely summable. 
Let $\varepsilon>0$. Since $M$ is uniformly absolutely summable, there is $\delta>0$ such that 
for all $\Omega\subset\N$ with $|\Omega|<\delta$ and all $y\in M$ it holds 
\[
 \sum_{n\in\Omega}\Bigl|\frac{1}{\lambda-q(n)}y_{n}\Bigr|
\leq C_{\lambda}\sum_{n\in\Omega}|y_{n}|
<C_{\lambda}\frac{\varepsilon}{C_{\lambda}}
=\varepsilon,
\]
yielding that $M_{\lambda}$ is uniformly absolutely summable. 
Thus $M_{\lambda}$ is $\sigma(\ell^{1},\ell^{\infty})$-compact. 
Further, it is easy to check that $M_{\lambda}$ is absolutely convex because $M$ is absolutely convex. 
Hence $R(\lambda,A)$ is $\mu(\ell^{\infty},\ell^{1})$-continuous by \eqref{eq:ex_resolv_bullet} 
for all $\lambda\in\rho(A)$.
Therefore \prettyref{ex:mackey_mazur_examples} (b) and \prettyref{prop:resolv_bullet} (b) yield 
$R(\lambda,A)^{\bullet}X^{\bullet}\subset D(A^{\bullet})$ for all $\lambda\in\rho(A)$.
We conclude that $R(\lambda,A)^{\bullet}=R(\lambda,A^{\bullet})$ for all $\lambda\in\rho(A)$ and
\[
\rho(A)=\rho(A^{\bullet})=\C\setminus\overline{q(\N)}
\]
by \prettyref{thm:resolv_bullet} (a) and (b).
\end{exa}

Our interest in the example above comes from \cite[Example 2.3, p.~147]{JacoSchwWint22} 
(and its role in \cite{JacoSchwWint22}) where we replaced the space $c_{0}$ by $\ell^{\infty}$. 
Next, we generalise \cite[Proposition 2.1.1, p.~19]{vanneerven1992}.

\begin{prop}\label{prop:weakly_bi_sun_closed}
Let $(X,\|\cdot\|,\tau)$ be a sequentially complete d-consistent Mazur--Saks space  
and $(T(t))_{t\geq 0}$ a bi-continuous semigroup on $X$ with generator $(A,D(A))$. 
For $G\subset X$ and $t>0$ we set $G_{0}\coloneqq G$ and 
$
G_{t}\coloneqq\{\frac{1}{t}\int_{0}^{t}T(s)g\d s\;|\;g\in G\}
$.
Then we have 
\[
\overline{G}^{\sigma(X,X^{\bullet})}\subset \bigcap_{t>0}\overline{\bigcup_{0\leq r\leq t}G_{r}}^{\sigma(X,X^{\circ})}.
\] 
In particular, if $G=\bigcap_{t>0}\overline{\bigcup_{0\leq r\leq t}G_{r}}^{\sigma(X,X^{\circ})}$, then $G$ is 
$\sigma(X,X^{\bullet})$-closed.
\end{prop}
\begin{proof}
Let $x\notin\bigcap_{t>0}\overline{\bigcup_{0\leq r\leq t}G_{r}}^{\sigma(X,X^{\circ})}$. We have to show that 
$x\notin \overline{G}^{\sigma(X,X^{\bullet})}$. By assumption there is some $t_{0}>0$ such that 
$x\notin \overline{\bigcup_{0\leq r\leq t_{0}}G_{r}}^{\sigma(X,X^{\circ})}$. 
Since the complement of the latter set is $\sigma(X,X^{\circ})$-open, there are some $n\in\N$ and 
$x_{i}^{\circ}\in X^{\circ}$, $1\leq i\leq n$, and $\varepsilon>0$ such that the $\sigma(X,X^{\circ})$-neighbourhood 
$V$ of $x$ given by 
\[
V\coloneqq V(x_{1}^{\circ},\ldots,x_{n}^{\circ};\varepsilon;x)\coloneqq\{y\in X\;|\;\forall\;1\leq i\leq n:\;
|\langle x_{i}^{\circ},x-y\rangle|<\varepsilon\}
\]
is disjoint from $\overline{\bigcup_{0\leq r\leq t_{0}}G_{r}}^{\sigma(X,X^{\circ})}$. 

Since $X^{\circ}=(X,\gamma)'$ by \prettyref{rem:bi_cont_dual}, for every $1\leq i\leq n$ there are $C_{i}>0$ 
and $p_{\gamma_{i}}\in \mathcal{P}_{\gamma}$ such that $|\langle x_{i}^{\circ},z\rangle|\leq C_{i}p_{\gamma_{i}}(z)$ 
for all $z\in X$ where $\mathcal{P}_{\gamma}$ is a directed system of seminorms that generates 
the mixed topology $\gamma$. From $\mathcal{P}_{\gamma}$ being directed it follows 
that there are $C\geq 1$ and $p_{\gamma}\in\mathcal{P}_{\gamma}$ such that 
$|\langle x_{i}^{\circ},z\rangle|\leq Cp_{\gamma}(z)$ for all $z\in X$ and $1\leq i\leq n$.
By the proof of \prettyref{thm:eqiv_norms} we know 
that $\gamma$-$\lim_{t\to 0\rlim}\frac{1}{t}\int_{0}^{t}T(s)x\d s-x=0$. Thus there 
is some $0<t_{1}\leq t_{0}$ such that 
\begin{equation}\label{eq:weakly_bi_sun_closed}
p_{\gamma}\Bigl(\frac{1}{t_{1}}\int_{0}^{t_{1}}T(s)x\d s-x\Bigr) < \frac{\varepsilon}{2C}.
\end{equation}
We claim that $\widetilde{V}\cap G=\varnothing$ where 
\[
\widetilde{V}\coloneqq V\Bigl(\frac{1}{t_{1}}\int_{0}^{t_{1}}T^{\circ}(s)x_{1}^{\circ}\d s,\ldots,
\frac{1}{t_{1}}\int_{0}^{t_{1}}T^{\circ}(s)x_{n}^{\circ}\d s;\frac{\varepsilon}{2};x\Bigr).
\]
Indeed, for $g\in G$ there is some $1\leq i_{0}\leq n$ such that 
\[
\Bigl|\langle x_{i_{0}}^{\circ},x-\frac{1}{t_{1}}\int_{0}^{t_{1}}T(s)g\d s\rangle\Bigr|\geq \varepsilon
\]
because $V\cap G_{t_{1}}=\varnothing$. Then we have 
\begin{align*}
  \Bigl|\langle\frac{1}{t_{1}}\int_{0}^{t_{1}}T^{\circ}(s)x_{i_{0}}^{\circ}\d s,x-g\rangle\Bigr|
&=\Bigl|\langle x_{i_{0}}^{\circ}, \frac{1}{t_{1}}\int_{0}^{t_{1}}T(s)x\d s
                                  -\frac{1}{t_{1}}\int_{0}^{t_{1}}T(s)g\d s\rangle\Bigr|\\
&\geq \Bigl|\langle x_{i_{0}}^{\circ},x-\frac{1}{t_{1}}\int_{0}^{t_{1}}T(s)g\d s\rangle\Bigr|
     -\Bigl|\langle x_{i_{0}}^{\circ},\frac{1}{t_{1}}\int_{0}^{t_{1}}T(s)x\d s-x\rangle\Bigr|\\
&\geq \varepsilon -Cp_{\gamma}\Bigl(\frac{1}{t_{1}}\int_{0}^{t_{1}}T(s)x\d s-x\Bigr)
 \geq \varepsilon -C\frac{\varepsilon}{2C}=\frac{\varepsilon}{2},
\end{align*}
which shows that $\widetilde{V}\cap G=\varnothing$ and proves the claim. 
However, $\frac{1}{t_{1}}\int_{0}^{t_{1}}T^{\circ}(s)x_{i}^{\circ}\d s\in D(A^{\circ})\subset X^{\bullet}$ 
for all $1\leq i\leq n$ by \prettyref{thm:generator} (d) and \prettyref{cor:bi_sun_dual}. 
Thus $\widetilde{V}$ is $\sigma(X,X^{\bullet})$-open and 
$\widetilde{V}\cap \overline{G}^{\sigma(X,X^{\bullet})}=\varnothing$. Due to \eqref{eq:weakly_bi_sun_closed} 
we have $x\in\widetilde{V}$ and hence $x\notin\overline{G}^{\sigma(X,X^{\bullet})}$.
\end{proof}

Now, we generalise the definition of (weak) equicontinuity w.r.t.~a norm-strongly continuous semigroup 
from \cite[p.~25]{vanneerven1992} and \cite[Proposition 2.2.2, p.~26]{vanneerven1992} to the bi-continuous setting.

\begin{defn}\label{defn:T(t)_equicont}
Let $(X,\|\cdot\|,\tau)$ be a sequentially complete d-consistent Saks space and $(T(t))_{t\geq 0}$ a bi-continuous semigroup on $X$. 
We say that a set $G\subset X$ is $\gamma$-$(T(t))_{t\geq 0}$\emph{-equicontinuous} if the set
$\{t\mapsto T(t)g\;|\;g\in G\}$ is $\gamma$-equicontinuous at $t=0$. 
We say that $G$ is $\sigma(X,X^{\circ})$-$(T(t))_{t\geq 0}$\emph{-equicontinuous} if for each $x^{\circ}\in X^{\circ}$
the set $\{t\mapsto \langle x^{\circ}, T(t)g\rangle\;|\;g\in G\}$ is equicontinuous at $t=0$. 
\end{defn}

\begin{rem}\label{rem:T(t)_equicont}
Let $(X,\|\cdot\|,\tau)$ be a sequentially complete d-consistent Saks space, $G\subset X$ and 
$(T(t))_{t\geq 0}$ a bi-continuous semigroup on $X$.
\begin{enumerate}
\item[(a)] If $(X,\gamma)$ is a Mazur space and $G$ $\gamma$-$(T(t))_{t\geq 0}$-equicontinuous, 
then $G$ is $\sigma(X,X^{\circ})$-$(T(t))_{t\geq 0}$-equicontinuous as $X^{\circ}=(X,\gamma)'$ 
by \prettyref{rem:bi_cont_dual}.
\item[(b)] If $G$ is $\gamma$-$(T(t))_{t\geq 0}$-equicontinuous, then $\overline{G}^{\gamma}$ is 
$\gamma$-$(T(t))_{t\geq 0}$-equicontinuous which is easily seen.
\item[(c)] If $G$ is $\sigma(X,X^{\circ})$-$(T(t))_{t\geq 0}$-equicontinuous, 
then $\overline{G}^{\sigma(X,X^{\circ})}$ is 
$\sigma(X,X^{\circ})$-$(T(t))_{t\geq 0}$-equicontinuous which is easily seen as well.
\end{enumerate}
\end{rem}

\begin{prop}\label{prop:T(t)_equicont}
Let $(X,\|\cdot\|,\tau)$ be a sequentially complete d-consistent Mazur--Saks space   
and $(T(t))_{t\geq 0}$ a bi-continuous semigroup on $X$. If $G\subset X$ is 
$\sigma(X,X^{\circ})$-$(T(t))_{t\geq 0}$-equicontinuous, then 
\[
\overline{G}^{\sigma(X,X^{\circ})}=\bigcap_{t>0}\overline{\bigcup_{0\leq r\leq t}G_{r}}^{\sigma(X,X^{\circ})}.
\] 
\end{prop}
\begin{proof}
The inclusion $\subset$ is clear since $G_{0}=G$. We prove the converse inclusion $\supset$ by contraposition. 
Let $x\notin \overline{G}^{\sigma(X,X^{\circ})}$.
We have to show that there is some $t_{0}>0$ such that 
$x\notin \overline{\bigcup_{0\leq r\leq t_{0}}G_{r}}^{\sigma(X,X^{\circ})}$. 
Like in \prettyref{prop:weakly_bi_sun_closed} there are some $n\in\N$ and 
$x_{i}^{\circ}\in X^{\circ}$, $1\leq i\leq n$, and $\varepsilon>0$ such that the $\sigma(X,X^{\circ})$-neighbourhood 
$V$ of $x$ given by 
\[
V\coloneqq V(x_{1}^{\circ},\ldots,x_{n}^{\circ};\varepsilon;x)\coloneqq\{y\in X\;|\;\forall\;1\leq i\leq n:\;
|\langle x_{i}^{\circ},x-y\rangle|<\varepsilon\}
\]
is disjoint from $G=G_{0}$. 

By the $\sigma(X,X^{\circ})$-$(T(t))_{t\geq 0}$-equicontinuity there is $t_{0}>0$ such that
for every $0\leq r\leq t_{0}$, $g\in G$ and $1\leq i\leq n$ we have
\[
|\langle x_{i}^{\circ},T(r)g-g\rangle|<\frac{\varepsilon}{2}.
\] 
This yields for every $0< r\leq t_{0}$, $g\in G$ and $1\leq i\leq n$ that
\[
    \Bigl|\langle x_{i}^{\circ}, \frac{1}{r}\int_{0}^{r}T(s)g\d s-g\rangle\Bigr|
\underset{\eqref{eq:eqiv_norms}}{\leq}\frac{\varepsilon}{2}.
\] 
We derive that for every $0< r\leq t_{0}$, $g\in G$ and $1\leq i\leq n$
\[
     \Bigl|\langle x_{i}^{\circ}, x-\frac{1}{r}\int_{0}^{r}T(s)g\d s\rangle\Bigr|
\geq |\langle x_{i}^{\circ}, x-g\rangle|-\Bigl|\langle x_{i}^{\circ}, g-\frac{1}{r}\int_{0}^{r}T(s)g\d s\rangle\Bigr|
\geq \varepsilon-\frac{\varepsilon}{2}=\frac{\varepsilon}{2}.
\]
We deduce that $\widetilde{V}\cap G_{r}=\varnothing$ for all $0< r\leq t_{0}$ where 
$\widetilde{V}\coloneqq V(x_{1}^{\circ},\ldots,x_{n}^{\circ};\frac{\varepsilon}{2};x)$. 
Since $\widetilde{V}\subset V$ and $V\cap G_{0}=\varnothing$, we obtain $\widetilde{V}\cap G_{0}=\varnothing$ as well. 
The set $\widetilde{V}$ is $\sigma(X,X^{\circ})$-open, which implies 
\[
\widetilde{V}\cap\overline{\bigcup_{0\leq r\leq t_{0}}G_{r}}^{\sigma(X,X^{\circ})}=\varnothing.
\]
This finishes the proof because $x\in\widetilde{V}$. 
\end{proof}

As a direct consequence of \prettyref{prop:weakly_bi_sun_closed} and \prettyref{prop:T(t)_equicont} we obtain 
the following corollary, which generalises \cite[Corollary 2.2.3, p.~26]{vanneerven1992}.

\begin{cor}\label{cor:weakly_bi_sun_closed}
Let $(X,\|\cdot\|,\tau)$ be a sequentially complete d-consistent Mazur--Saks space 
and $(T(t))_{t\geq 0}$ a bi-continuous semigroup on $X$. Then the $\sigma(X,X^{\circ})$- and 
the $\sigma(X,X^{\bullet})$-closure of $\sigma(X,X^{\circ})$-$(T(t))_{t\geq 0}$-equicontinuous sets coincide. 
In particular, $\sigma(X,X^{\circ})$-closed $\sigma(X,X^{\circ})$-$(T(t))_{t\geq 0}$-equicontinuous sets are 
$\sigma(X,X^{\bullet})$-closed.
\end{cor}

The following two corollaries represent \cite[Corollaries 2.2.4, 2.2.5 p.~26]{vanneerven1992} 
in the bi-continuous setting.

\begin{cor}\label{cor:weakly_bi_sun_relative_closed}
Let $(X,\|\cdot\|,\tau)$ be a sequentially complete d-consistent Mazur--Saks space   
and $(T(t))_{t\geq 0}$ a bi-continuous semigroup on $X$. 
Then the relative $\sigma(X,X^{\circ})$- and $\sigma(X,X^{\bullet})$-topology coincide on 
$\sigma(X,X^{\circ})$-$(T(t))_{t\geq 0}$-equicontinuous sets.
\end{cor}
\begin{proof}
Let $G\subset X$ be $\sigma(X,X^{\circ})$-$(T(t))_{t\geq 0}$-equicontinuous and suppose that $H\subset G$ 
is relatively $\sigma(X,X^{\circ})$-closed. Denoting by $\widetilde{H}$ the $\sigma(X,X^{\circ})$-closure of 
$H$ in $X$, we have $\widetilde{H}\cap G=H$. 
Further, $\widetilde{H}$ is $\sigma(X,X^{\circ})$-$(T(t))_{t\geq 0}$-equicontinuous 
by \prettyref{rem:T(t)_equicont} (c) and thus $\sigma(X,X^{\bullet})$-closed 
by \prettyref{cor:weakly_bi_sun_closed}, yielding that $H=\widetilde{H}\cap G$ is relatively 
$\sigma(X,X^{\bullet})$-closed in $G$.
\end{proof}

\begin{cor}\label{cor:weakly_bi_sun_convergence}
Let $(X,\|\cdot\|,\tau)$ be a sequentially complete d-consistent Mazur--Saks space   
and $(T(t))_{t\geq 0}$ a bi-continuous semigroup on $X$. 
Then a $\sigma(X,X^{\circ})$-$(T(t))_{t\geq 0}$-equi\-continuous sequence is $\sigma(X,X^{\circ})$-convergent if 
and only if it is $\sigma(X,X^{\bullet})$-conver\-gent.
\end{cor}
\begin{proof}
The implication $\Rightarrow$ is obvious because $\sigma(X,X^{\circ})$ is a finer topology 
than $\sigma(X,X^{\bullet})$. Let us turn to the implication $\Leftarrow$. 
Let $(x_{n})_{n\in\N}$ be a $\sigma(X,X^{\circ})$-$(T(t))_{t\geq 0}$-equicontinuous sequence in $X$ that 
is $\sigma(X,X^{\bullet})$-convergent to some $x\in X$. Then the set $G\coloneqq\{x_{n}\;|\;n\in\N\}\cup\{x\}$ 
is the $\sigma(X,X^{\bullet})$-closure of $\{x_{n}\;|\;n\in\N\}$ and so its $\sigma(X,X^{\circ})$-closure 
by \prettyref{cor:weakly_bi_sun_closed} as well. Hence $G$ is also 
$\sigma(X,X^{\circ})$-$(T(t))_{t\geq 0}$-equicontinuous by \prettyref{rem:T(t)_equicont} (c). 
Let $V$ be a $\sigma(X,X^{\circ})$-open neighbourhood of $x$ in $X$. Then $V\cap G$ is relatively 
$\sigma(X,X^{\circ})$-open in $G$ and thus relatively $\sigma(X,X^{\bullet})$-open in $G$ by 
\prettyref{cor:weakly_bi_sun_relative_closed}. This implies that all but finitely many $x_{n}$ lie in 
$(V\cap G)\subset V$, which we had to show.  
\end{proof}

Now, we give a class of sets to which the three preceding corollaries can be applied 
due to \prettyref{rem:T(t)_equicont} (a) if $(X,\gamma)$ is a Mazur space.

\begin{prop}\label{prop:resolv_equicont}
Let $(X,\|\cdot\|,\tau)$ be a sequentially complete d-consistent Saks space
and $(T(t))_{t\geq 0}$ a bi-continuous semigroup on $X$ with generator $(A,D(A))$. 
If $H$ is $\|\cdot\|$-bounded, then $R(\lambda,A)H$ is 
$\gamma$-$(T(t))_{t\geq 0}$-equicontinuous for all $\lambda\in\rho(A)$.
\end{prop}
\begin{proof}
Let $\mathcal{P}_{\gamma}$ be a directed system of seminorms that generates the mixed topology $\gamma$. 
Due to \cite[Lemma 5.5 (a), p.~2680]{kruse_meichnser_seifert2018} and 
\cite[Remark 2.3 (c), p.~3]{kruse_schwenninger2022} we may choose $\mathcal{P}_{\gamma}$ such that $\|x\|=\sup_{p_{\gamma}\in\mathcal{P}_{\gamma}}p_{\gamma}(x)$ 
for all $x\in X$.
We start with noting that the map $s\mapsto T(s)AR(\lambda,A)h$ is $\gamma$-Pettis integrable on $[0,t]$ and 
\[
T(t)R(\lambda,A)h-R(\lambda,A)h=\int_{0}^{t}T(s)AR(\lambda,A)h\d s
\]
for all $t>0$ and $h\in H$ by \prettyref{thm:generator} (c) and (d). 
For any $x'\in (X,\gamma)'$ we get
\[
\Bigl|\langle x',\int_{0}^{t}T(s)AR(\lambda,A)h\d s\rangle\Bigr|
\leq t\sup_{s\in[0,t]}|\langle x', T(s)AR(\lambda,A)h\rangle|,
\]
resulting in 
\begin{align*}
     p_{\gamma}\Bigl(\int_{0}^{t}T(s)AR(\lambda,A)h\d s\Bigr)
&\leq t\sup_{s\in[0,t]} p_{\gamma}(T(s)AR(\lambda,A)h)
 \leq t\sup_{s\in[0,t]} \|T(s)AR(\lambda,A)h\|\\
&\leq t\sup_{s\in[0,t]} \|T(s)\|_{\mathcal{L}(X)}\|AR(\lambda,A)h\|\\
&\leq tM\e^{|\omega|t}\|AR(\lambda,A)\|_{\mathcal{L}(X)}\|h\|
\end{align*}
for any $p_{\gamma}\in\mathcal{P}_{\gamma}$ since $(T(t))_{t\geq 0}$ is exponentially bounded and 
$AR(\lambda,A)\in\mathcal{L}(X)$ because $AR(\lambda,A)x=\lambda R(\lambda,A)x-x$ for all $x\in X$. 
Since $H$ is $\|\cdot\|$-bounded, there is $C>0$ such that $\|h\|\leq C$ for all $h\in H$, which yields
\[
p_{\gamma}(T(t)R(\lambda,A)h-R(\lambda,A)h)\leq tMC\e^{|\omega|t}\|AR(\lambda,A)\|_{\mathcal{L}(X)}
\]
for all $t>0$ and $p_{\gamma}\in\mathcal{P}_{\gamma}$. This means that $R(\lambda,A)H$ 
is $\gamma$-$(T(t))_{t\geq 0}$-equicontinuous at $t=0$.
\end{proof}

\prettyref{prop:resolv_equicont} in combination with \prettyref{rem:T(t)_equicont} (b) 
generalises \cite[Proposition 2.2.6, p.~27]{vanneerven1992}.
The next proposition transfers one direction of \cite[Corollary 2.2.8, p.~28]{vanneerven1992} to the bi-continuous 
setting. 

\begin{prop}\label{prop:resolvent_weakly_compact}
Let $(X,\|\cdot\|,\tau)$ be a sequentially complete d-consistent Mazur--Saks space
and $(T(t))_{t\geq 0}$ a bi-continuous semigroup on $X$ with generator $(A,D(A))$. 
Let $G\subset X$ be $\sigma(X,X^{\bullet})$-compact. Then the following assertions hold: 
\begin{enumerate}
\item[(a)] $G$ is $\|\cdot\|$-bounded.
\item[(b)] If $\lambda\in\rho(A)$ is such that $R(\lambda,A)^{\bullet}X^{\bullet}\subset X^{\bullet}$, 
then $R(\lambda,A)G$ is $\sigma(X,X^{\circ})$-compact. 
In particular, $R(\lambda,A)G$ is $\sigma(X,X^{\circ})$-compact if $\re\lambda >\omega_{0}(T)$ or 
$R(\lambda,A)$ is $\gamma$-continuous. 
\end{enumerate}
\end{prop}
\begin{proof}
(a) Let $G\subset X$ be $\sigma(X,X^{\bullet})$-compact. We may regard $G$ as a subset of ${X^{\bullet}}'$
via the the canonical map $j\colon X\to {X^{\bullet}}'$ from \prettyref{cor:embed_bi_sun_dual}. 
Then $G$ is $\sigma({X^{\bullet}}',X^{\bullet})$-compact and thus $\|\cdot\|_{{X^{\bullet}}'}$-bounded 
by the uniform boundedness principle, implying the $\|\cdot\|$-boundedness by \prettyref{cor:embed_bi_sun_dual}.

(b) The resolvent map $R(\lambda,A)$ is $\sigma(X,X^{\bullet})$-continuous since 
$R(\lambda,A)^{\bullet}X^{\bullet}\subset X^{\bullet}$ by assumption and 
\[
\langle x^{\bullet}\circ R(\lambda,A), x\rangle = \langle R(\lambda,A)^{\bullet}x^{\bullet}, x\rangle
\]
for all $x^{\bullet}\in X^{\bullet}$ and $x\in X$. So $R(\lambda,A)G$ is $\sigma(X,X^{\bullet})$-compact. 
Due to the $\|\cdot\|$-boundedness of $G$ by part (a), \prettyref{prop:resolv_equicont} 
and \prettyref{rem:T(t)_equicont} (a) we have that $R(\lambda,A)G$ is
$\sigma(X,X^{\circ})$-$(T(t))_{t\geq 0}$-equicontinuous. We conclude 
that $R(\lambda,A)G$ is $\sigma(X,X^{\circ})$-compact by \prettyref{cor:weakly_bi_sun_relative_closed}. 

The rest of statement (b) is a consequence of \prettyref{prop:resolv_bullet} (a) and (b).
\end{proof}

\section{Bi-sun reflexivity}
\label{sect:reflexivity}

We recall from \prettyref{cor:embed_bi_sun_dual} that the canonical map 
$j\colon X\to {X^{\bullet}}'$ given by $\langle j(x), x^{\bullet}\rangle\coloneqq\langle x^{\bullet},x\rangle$ 
is injective and $j(X_{\operatorname{cont}})= X^{\bullet\bullet}\cap j(X)$ holds (under the assumptions of 
\prettyref{cor:embed_bi_sun_dual}). This leads to the following generalisation of 
$\odot$-reflexivity w.r.t.~a semigroup (see \cite[p.~7]{vanneerven1992}). 

\begin{defn}
Let $(X,\|\cdot\|,\tau)$ be a sequentially complete d-consistent Mazur--Saks space
and $(T(t))_{t\geq 0}$ a bi-continuous semigroup on $X$. 
We say that $X$ is $\bullet$\emph{-reflexive (or bi-sun reflexive) w.r.t.}~$(T(t))_{t\geq 0}$ if 
$j(X_{\operatorname{cont}})= X^{\bullet\bullet}$.
\end{defn}

\begin{rem}
\begin{enumerate}
\item[(a)] Let $(X,\|\cdot\|)$ be a Banach space. For a $\|\cdot\|$-strongly continuous semigroup $(T(t))_{t\geq 0}$ on $X$ 
we have $X_{\operatorname{cont}}=X$ and $X^{\bullet\bullet}=X^{\odot\odot}$. Thus $X$ is 
$\bullet$-reflexive w.r.t.~$(T(t))_{t\geq 0}$ if and only if it is $\odot$-reflexive w.r.t.~$(T(t))_{t\geq 0}$.
\item[(b)] One might object to coining the property 
$j(X_{\operatorname{cont}})= X^{\bullet\bullet}$ by ``$\bullet$-reflexivity'', as it is not symmetric. 
However, our main point in studying this property lies in its value for describing the Favard space 
$Fav(T)$ and its relation to the generator $(A,D(A))$ of $(T(t))_{t\geq 0}$ (and by part (a), it is indeed a reasonable name for this property). 
\end{enumerate}
\end{rem} 

First, we study the relation between a bi-continuous semigroup and its restriction to its space of strong continuity with regard to 
(bi-)sun reflexivity.

\begin{prop}\label{prop:relation_sun_bi_sun}
Let $(X,\|\cdot\|,\tau)$ be a sequentially complete d-consistent Mazur--Saks space and 
$(T(t))_{t\geq 0}$ a bi-continuous semigroup on $X$. Then the following assertions hold:
\begin{enumerate}
\item[(a)] $T^{\bullet\bullet}j(x)=j(T(t)x)$ for all $t\geq 0$ and $x\in X_{\operatorname{cont}}$.
\item[(b)] The maps $\iota\colon X^{\bullet}\to X_{\operatorname{cont}}^{\odot}$, $\iota(x^{\bullet})\coloneqq 
x^{\bullet}_{\mid X_{\operatorname{cont}}}$, and $\kappa\colon X_{\operatorname{cont}}^{\odot\odot}\to X^{\bullet\bullet}$, 
$\kappa(y)\coloneqq y\circ\iota$, are well-defined, linear and continuous, and $\iota$ is injective. 
In particular, we have the continuous embeddings $X^{\bullet}\hookrightarrow X_{\operatorname{cont}}^{\odot}$ 
and $(X_{\operatorname{cont}}^{\odot\odot}/\ker (\kappa))\hookrightarrow  X^{\bullet\bullet}$.
\item[(c)] $\kappa\circ j_{0}= j$ on $X_{\operatorname{cont}}$ where 
$j_{0}\colon X_{\operatorname{cont}}\to {X_{\operatorname{cont}}^{\odot}}'$ is the canonical map given 
by $\langle j_{0}(x), x^{\odot}\rangle\coloneqq\langle x^{\odot},x\rangle$.
\item[(d)] If $X$ is $\bullet$-reflexive w.r.t.~$(T(t))_{t\geq 0}$, then $\kappa$ is surjective.
\item[(e)] If $\iota$ is surjective, then $\kappa$ is injective.
\end{enumerate}
\end{prop}
\begin{proof}
(a) We note that $j(X_{\operatorname{cont}})\subset X^{\bullet\bullet}$ 
by \prettyref{cor:embed_bi_sun_dual} and $T(t)X_{\operatorname{cont}}\subset X_{\operatorname{cont}}$ 
for all $t\geq 0$ by \prettyref{thm:generator} (g), 
which implies
\[
 \langle T^{\bullet\bullet}(t)j(x),x^{\bullet} \rangle
=\langle j(x),T^{\bullet}(t)x^{\bullet} \rangle
=\langle T^{\bullet}(t)x^{\bullet},x \rangle
=\langle x^{\bullet},T(t)x \rangle
=\langle j(T(t)x),x^{\bullet}\rangle
\]
for any $t\geq 0$, $x\in X_{\operatorname{cont}}$ and $x^{\bullet}\in X^{\bullet}$. 

(b) Due \prettyref{thm:generator} (g) and \prettyref{rem:bi_cont_dual} 
$X_{\operatorname{cont}}$ is sequentially $\gamma$-dense in $X$ and $X^{\circ}=X_{\operatorname{seq-}\gamma}'$.
Thus the continuous linear map $\iota_{0}\colon (X^{\circ},\|\cdot\|_{X^{\circ}})\to 
(X_{\operatorname{cont}}',\|\cdot\|_{X_{\operatorname{cont}}'})$, 
$x^{\circ}\mapsto x^{\circ}_{\mid X_{\operatorname{cont}}}$, 
is injective and we note that $\iota={\iota_{0}}_{\mid X^{\bullet}}$. From $T^{\circ}(t)x^{\circ}=T'(t)x^{\circ}$ for all $t\geq 0$ 
and $x^{\circ}\in X^{\circ}$ it follows $\iota_{0}(X^{\bullet})\subset X_{\operatorname{cont}}^{\odot}$.
Thus we get $y\circ\iota\in{X^{\bullet}}'$ for any 
$y\in {X_{\operatorname{cont}}^{\odot}}'$ and 
\begin{align*}
 \langle {T^{\bullet}}'(t)(y\circ\iota),x^{\bullet} \rangle
&=\langle y,T^{\bullet}(t)x^{\bullet} \rangle
 =\langle y,T^{\circ}(t)x^{\bullet} \rangle
 =\langle y,T'(t)x^{\bullet} \rangle
 =\langle y,(T_{\mid X_{\operatorname{cont}}})^{\odot}(t)x^{\bullet} \rangle\\
&=\langle {(T_{\mid X_{\operatorname{cont}}})^{\odot}}'(t)y,x^{\bullet} \rangle
\end{align*}
for any $t\geq 0$, $y\in {X_{\operatorname{cont}}^{\odot}}'$ and $x^{\bullet}\in X^{\bullet}$, 
implying $\kappa(y)\in X^{\bullet\bullet}$ for all $y\in X_{\operatorname{cont}}^{\odot\odot}$. 
Further, the estimate $\|\iota(x^{\bullet})\|_{X_{\operatorname{cont}}^{\odot}}\leq \|x^{\bullet}\|_{X^{\bullet}}$ for all 
$x^{\bullet}\in X^{\bullet}$ yields $\|\kappa(y)\|_{X^{\bullet\bullet}}\leq \|y\|_{X_{\operatorname{cont}}^{\odot\odot}}$ 
for all $y\in X_{\operatorname{cont}}^{\odot\odot}$, which finishes the proof of part (b).

(c) We note that $j_{0}(X_{\operatorname{cont}}^{\phantom{\odot\odot}})\subset X_{\operatorname{cont}}^{\odot\odot}$ 
by \cite[p.~7]{vanneerven1992}. Let $x\in X_{\operatorname{cont}}$. Then we have $j_{0}(x)\in  X_{\operatorname{cont}}^{\odot\odot}$ 
and
\[
\langle \kappa(j_{0}(x)), x^{\bullet}\rangle
=\langle j_{0}(x),\iota(x^{\bullet})\rangle
=\langle \iota(x^{\bullet}),x\rangle
=\langle x^{\bullet},x\rangle
=\langle j(x),x^{\bullet}\rangle
\]
for all $x^{\bullet}\in X^{\bullet}$.

(d) This follows from (c) since $X^{\bullet\bullet}=j(X_{\operatorname{cont}})$ and 
$j_{0}(X_{\operatorname{cont}}^{\phantom{\odot\odot}})\subset X_{\operatorname{cont}}^{\odot\odot}$.

(e) If $\iota$ is surjective, then $\iota(X^{\bullet})= X_{\operatorname{cont}}^{\odot}$ and thus $\ker (\kappa)=\{0\}$. 
\end{proof}

\begin{prop}\label{prop:relation_reflexivities_2}
Let $(X,\|\cdot\|,\tau)$ be a sequentially complete d-consistent Mazur--Saks space and 
$(T(t))_{t\geq 0}$ a bi-continuous semigroup on $X$.
If $X$ is $\bullet$-reflexive w.r.t.~$(T(t))_{t\geq 0}$ and $\iota(X^{\bullet})= X_{\operatorname{cont}}^{\odot}$ 
with $\iota\colon X^{\bullet}\to X_{\operatorname{cont}}^{\odot}$ from \prettyref{prop:relation_sun_bi_sun} (b), 
then $X_{\operatorname{cont}}$ is $\odot$-reflexive w.r.t.~$(T(t)_{\mid X_{\operatorname{cont}}})_{t\geq 0}$, 
the map $\kappa\colon X_{\operatorname{cont}}^{\odot\odot}\to X^{\bullet\bullet}$ from 
\prettyref{prop:relation_sun_bi_sun} (b) is a topological isomorphism and 
\[
 \overline{D(A)}^{\|\cdot\|}
=X_{\operatorname{cont}}^{\phantom{\odot\odot}}
=X_{\operatorname{cont}}^{\odot\odot}
=X^{\bullet\bullet}
\]
where we identified $X_{\operatorname{cont}}^{\odot\odot}$ 
with a subspace of $X^{\bullet\bullet}$ via $\kappa$ and $X_{\operatorname{cont}}$ with a subspace
of $X^{\bullet\bullet}$ via the canonical map $j\colon X\to {X^{\bullet}}'$, which fulfils 
$j=\kappa\circ j_{0}$ with the canonical map $j_{0}\colon X_{\operatorname{cont}}\to {X_{\operatorname{cont}}^{\odot}}'$ 
by \prettyref{prop:relation_sun_bi_sun} (c).
\end{prop}
\begin{proof} 
First, we show that $j_{0}(X_{\operatorname{cont}}^{\phantom{\odot\odot}})=X_{\operatorname{cont}}^{\odot\odot}$. 
Let $y\in X_{\operatorname{cont}}^{\odot\odot}$. Then there is $x\in X_{\operatorname{cont}}$ such that 
$\kappa(y)=j(x)$ since $X$ is $\bullet$-reflexive. 
For any $x^{\odot}\in X_{\operatorname{cont}}^{\odot}$ there exists $x^{\bullet}\in X^{\bullet}$ such that 
$\iota(x^{\bullet})=x^{\odot}$ because $\iota(X^{\bullet})=X_{\operatorname{cont}}^{\odot}$. Hence we get 
\begin{align*}
 \langle j_{0}(x), x^{\odot}\rangle
&=\langle x^{\odot},x\rangle
 =\langle\iota(x^{\bullet}),x\rangle
 =\langle x^{\bullet},x\rangle
 =\langle j(x), x^{\bullet}\rangle
 =\langle\kappa(y),x^{\bullet}\rangle
 =\langle y, \iota(x^{\bullet})\rangle\\
&=\langle y, x^{\odot}\rangle
\end{align*}
for all $x^{\odot}\in X_{\operatorname{cont}}^{\odot}$, implying that 
$y=j_{0}(x)\in j_{0}(X_{\operatorname{cont}}^{\phantom{\odot\odot}})$ and so the $\odot$-reflexivity of $X_{\operatorname{cont}}$ 
since $j_{0}(X_{\operatorname{cont}}^{\phantom{\odot\odot}})\subset X_{\operatorname{cont}}^{\odot\odot}$ always holds.

Second, it follows from the open mapping theorem and \prettyref{prop:relation_sun_bi_sun} (b), (d) and (e) 
that $\kappa$ is a topological isomorphism. 
The observation $\overline{D(A)}^{\|\cdot\|}=X_{\operatorname{cont}}^{\phantom{\odot\odot}}$ by \prettyref{thm:generator} (g) finishes 
the proof.
\end{proof}

The next proposition generalises \cite[Corollary 1.3.2, p.~6]{vanneerven1992}, namely, that 
a reflexive Banach space $X$ is $\odot$-reflexive.

\begin{prop}\label{prop:relation_reflexivities_1}
Let $(X,\|\cdot\|,\tau)$ be a sequentially complete d-consistent Mazur--Saks space. 
If $(X,\gamma)$ is semi-reflexive, then $X^{\bullet}=X^{\circ}=(X,\gamma)'$, the canonical map $j\colon X\to {X^{\bullet}}'$ 
is surjective and $X$ is $\bullet$-reflexive w.r.t.~any bi-continuous semigroup $(T(t))_{t\geq 0}$ on $X$.
\end{prop}
\begin{proof}
The space $X^{\circ}=(X,\gamma)'$ is a closed subspace of the Banach space $(X',\|\cdot\|_{X'})$ 
by \prettyref{rem:bi_cont_dual}. Due to $(X,\gamma)$ being semi-reflexive, \cite[I.1.18 Proposition (i), p.~15]{cooper1978} 
and the Mackey--Arens theorem we have 
\[
(X^{\circ},\|\cdot\|_{X^{\circ}})'=X=(X^{\circ},\sigma(X^{\circ},X))'
\]
where $\|\cdot\|_{X^{\circ}}$ is the restriction of $\|\cdot\|_{X'}$ to $X^{\circ}$
We deduce that for the bi-sun dual $X^{\bullet}$ w.r.t.~a $\tau$-bi-continuous semigroup 
$(T(t))_{t\geq 0}$ it holds
\[
X^{\bullet}=\overline{X^{\bullet}}^{\|\cdot\|_{X^{\circ}}}=\overline{X^{\bullet}}^{\sigma(X^{\circ},X)}=X^{\circ}
\]
by \cite[8.2.5 Proposition, p.~149]{jarchow1981} because $X^{\bullet}$ is a $\|\cdot\|_{X^{\circ}}$-closed, 
$\sigma(X^{\circ},X)$-dense linear subspace of $X^{\circ}$ by \prettyref{cor:bi_sun_dual}. It follows that 
${X^{\bullet}}'={X^{\circ}}'=X$ since $(X,\gamma)$ is semi-reflexive, implying 
\[
X_{\operatorname{cont}}=X^{\bullet\bullet}\cap X=X^{\bullet\bullet}\cap {X^{\bullet}}'=X^{\bullet\bullet}
\]
by \prettyref{cor:embed_bi_sun_dual} where we identified $X_{\operatorname{cont}}$ and $X$ with subspaces 
of ${X^{\bullet}}'$ via $j$.
\end{proof}

Let $(X,\|\cdot\|)$ be a Banach space, $\tau$ a Hausdorff locally convex topology on $X$ which is 
coarser than the $\|\cdot\|$-topology, and let $\gamma\coloneqq\gamma(\|\cdot\|,\tau)$ be the mixed topology. 
Then the space $(X,\gamma)$ is semi-reflexive if and only if $B_{\|\cdot\|}$ is $\sigma(X,(X,\tau)')$-compact 
by \cite[I.1.21 Corollary, p.~16]{cooper1978}. Moreover, $(X,\gamma)$ is a semi-Montel space, 
thus semi-reflexive, if and only if $B_{\|\cdot\|}$ is $\tau$-compact by 
\cite[I.1.13 Proposition, p.~11]{cooper1978}. This second condition is fulfilled 
for the triple $(\mathrm{C}_{\operatorname{b}}(\Omega),\|\cdot\|_{\infty},\tau_{\operatorname{co}})$ 
from \prettyref{ex:mackey_mazur_examples} (b) if, in addition, 
$\Omega$ is discrete by \cite[II.1.24 Remark 4), p.~88--89]{cooper1978}. 
The first condition is fulfilled for the Saks spaces from \prettyref{ex:mackey_mazur_examples} (c), (d), (e) and (f). 
It is fulfilled in example (c) by \cite[V.2.6 Proposition, p.~234]{cooper1978} and in the latter examples 
since $(X_{0}',\mu(X_{0}',X_{0}))''=X_{0}'$ by the Mackey--Arens theorem and 
$(\mathcal{L}(H),\beta_{\operatorname{sot}^{\ast}})''=\mathcal{N}(H)'=\mathcal{L}(H)$ 
for any Banach space $X_{0}$ and any separable Hilbert space $H$. In combination with 
\prettyref{thm:generator} (g) and \prettyref{prop:relation_reflexivities_1} 
we obtain the following.

\begin{cor}\label{cor:relation_reflexivities}
$(X,\gamma)$ is a semi-reflexive Mackey--Mazur space where $\gamma\coloneqq\gamma(\|\cdot\|,\tau)$ is the mixed topology, 
$X$ is $\bullet$-reflexive w.r.t.~any bi-continuous semigroup $(T(t))_{t\geq 0}$ on $X$ 
with generator $(A,D(A))$ and 
\[
X^{\bullet}=X^{\circ}=(X,\gamma)'
\quad\text{as well as}\quad
 \overline{D(A)}^{\|\cdot\|}
=X_{\operatorname{cont}}
=X^{\bullet\bullet}
\]
for each of the triples $(X,\|\cdot\|,\tau)$ from \prettyref{ex:mackey_mazur_examples} (c), (d), (e), (f) and
\begin{enumerate}
\item[(a)] if $(X,\|\cdot\|)$ is reflexive,
\item[(b)] if $\Omega$ is discrete.
\end{enumerate}
\end{cor}

Due to \prettyref{ex:resolv_bullet}, \prettyref{prop:relation_reflexivities_2} and 
\prettyref{cor:relation_reflexivities} (b) we have the following example.

\begin{exa}\label{ex:bi_sun_reflexive} 
Let $q\colon \N\to\C\setminus\{0\}$ such that $\sup_{n\in\N}\re q(n)<\infty$ and $(\tfrac{1}{q(n)})_{n\in\N}\in c_{0}$, and 
let $(T(t))_{t\geq 0}$ be the bi-continuous multiplication semigroup on 
$(\ell^{\infty},\|\cdot\|_{\infty},\mu(\ell^{\infty},\ell^{1}))$ from \prettyref{ex:resolv_bullet} given by 
\[
T(t)x\coloneqq(\e^{q(n)t}x_{n})_{n\in\N},\quad x\in\ell^{\infty},\, t\geq 0,
\]
with generator
$
A\colon D(A)\to \ell^{\infty},\;Ax =qx,
$
on the domain 
\[
D(A)=\{x\in\ell^{\infty}\;|\;(q(n)x_{n})_{n\in\N}\in\ell^{\infty}\}. 
\]
Due to \cite[p.~354]{budde2019} we have $(\ell^{\infty})_{\operatorname{cont}}=c_{0}$ 
since $(\tfrac{1}{q(n)})_{n\in\N}\in c_{0}$.
Further, we note that $\ell^{\infty}=\mathrm{C}_{\operatorname{b}}(\N)$ is $\bullet$-reflexive 
w.r.t.~$(T(t))_{t\geq 0}$, $c_{0}$ is $\odot$-reflexive w.r.t.~$(T(t)_{\mid c_{0}})_{t\geq 0}$ and 
\[
(\ell^{\infty})^{\bullet}=(\ell^{\infty})^{\circ}=\mathrm{M}_{\operatorname{t}}(\N)=\ell^{1}
\;\;\text{as well as}\;\;
 \overline{D(A)}^{\|\cdot\|_{\infty}}
=(\ell^{\infty})_{\operatorname{cont}}^{\phantom{\odot\odot}}
=c_{0}
=c_{0}^{\odot\odot}
=(\ell^{\infty})^{\bullet\bullet}
\]
since $c_{0}^{\odot}=\ell^{1}=(\ell^{\infty})^{\bullet}$ by \cite[Chap.~I, 4.11 Proposition, p.~32]{engel_nagel2000}.
\end{exa}

In \cite[Example 1.3.10, p.~9]{vanneerven1992} it is observed that 
$c_{0}$ is $\odot$-reflexive w.r.t.~$(T(t)_{\mid c_{0}})_{t\geq 0}$ for $q(n)\coloneqq-n$, $n\in\N$.

\section{The Favard space}
\label{sect:favard}

We begin this section with the definition of the Favard space.

\begin{defn}\label{def:favard}
Let $(X,\|\cdot\|,\tau)$ be a sequentially complete Saks space and $(T(t))_{t\geq 0}$ a 
bi-continuous semigroup on $X$. Then the \emph{Favard space (class)} of $(T(t))_{t\geq 0}$ is 
defined by 
\[
Fav(T)\coloneqq\{x\in X\;|\;\limsup_{t\to 0\rlim}\frac{1}{t}\|T(t)x-x\|<\infty\}.
\]
\end{defn}

\begin{rem}\label{rem:favard}
Let $(X,\|\cdot\|,\tau)$ be a sequentially complete Saks space and $(T(t))_{t\geq 0}$ a 
bi-continuous semigroup on $X$.
\begin{enumerate}
\item[(a)] It is obvious from the definition of the generator $(A,D(A))$ that $D(A)\subset Fav(T)$. 
\item[(b)] From $\|T(t)x-x\|=t\tfrac{1}{t}\|T(t)x-x\|$ for all $t>0$ and $x\in X$, we obtain 
$Fav(T)\subset X_{\operatorname{cont}}$ where $X_{\operatorname{cont}}$ is the space of $\|\cdot\|$-strong 
continuity of $(T(t))_{t\geq 0}$ from \prettyref{thm:generator} (g).
\end{enumerate}
\end{rem}

Our goal is to characterise those bi-continuous semigroups on $X$ for which $Fav(T)=D(A)$ holds. 
A class of bi-continuous semigroups for which this holds are the dual semigroups of norm-strongly 
continuous semigroups. 

\begin{exa}\label{ex:favard_dual_semigroup}
Let $(X,\|\cdot\|)$ be a Banach space and $(S(t))_{t\geq 0}$ a $\|\cdot\|$-strongly continuous semigroup on $X$ 
with generator $(A,D(A))$. 
Then $(S'(t))_{t\geq 0}$ is a bi-continuous semigroup semigroup on $(X',\|\cdot\|_{X'},\sigma(X',X))$ 
by \cite[Proposition 3.18, p.~78]{kuehnemund2001} with generator $(A',D(A'))$ and 
\[
 Fav(S')=D(A')=Fav(S^{\bullet})=Fav(S^{\odot})
\]
by \cite[Theorem 1.2.3, p.~4]{vanneerven1992}, \cite[Theorem 3.2.1, p.~54]{vanneerven1992} 
and \cite[Corollary 3.2.2, p.~55]{vanneerven1992}. 
\end{exa}

We note the following generalisation of \cite[Chap.~II, Proposition, Corollary, p.~60--61]{engel_nagel2000} 
for restrictions of bi-continuous semigroups which helps to explain when the equation $Fav(T)=D(A)$ 
is inherited by restricted bi-continuous semigroups.

\begin{prop}\label{prop:domain_part_generator}
Let $(X,\|\cdot\|,\tau)$ be a sequentially complete Saks space, 
$(T(t))_{t\geq 0}$ a bi-continuous semigroup on $X$ with generator $(A,D(A))$,  
$Y$ a $(T(t))_{t\geq 0}$-invariant sequentially $\gamma$-closed linear subspace of $X$ 
and denote by $\|\cdot\|_{Y}$ and $\tau_{Y}$ the restrictions of $\|\cdot\|$ and $\tau$ to $Y$, respectively. 
Then the following assertions hold:
\begin{enumerate}
\item[(a)] The triple $(Y,\|\cdot\|_{Y},\tau_{Y})$ is a sequentially complete Saks space and 
$(T(t)_{\mid Y})_{t\geq 0}$ is a bi-continuous semigroup on $(Y,\|\cdot\|_{Y},\tau_{Y})$. 
\item[(b)] The generator of $(T(t)_{\mid Y})_{t\geq 0}$ is the part $A_{\mid Y}$ of $A$ in $Y$, i.e. 
\[
A_{\mid Y}y=Ay,\quad y\in Y,
\]
with domain
\[
D(A_{\mid Y})=D(A)\cap Y.
\]
\end{enumerate}
\end{prop}
\begin{proof}
(a) The triple $(Y,\|\cdot\|_{Y},\tau_{Y})$ is a sequentially complete Saks space because $(X,\|\cdot\|,\tau)$ 
is a sequentially complete Saks space and $Y$ a sequentially $\gamma$-closed, in particular $\|\cdot\|$-closed, 
linear subspace of $X$. 
Since $(T(t))_{t\geq 0}$ is a bi-continuous semigroup on $(X,\|\cdot\|,\tau)$, and $Y$ is $(T(t))_{t\geq 0}$-invariant, it 
follows from \cite[Definition 3, p.~207]{kuehnemund2003} that $(T(t)_{\mid Y})_{t\geq 0}$ is a bi-continuous semigroup on $(Y,\|\cdot\|_{Y},\tau_{Y})$.

(b) Let $(C,D(C))$ be the generator of $(T(t)_{\mid Y})_{t\geq 0}$. If $y\in D(C)\subset Y$, then 
\[
 \sup_{t\in (0,1]}\Bigl\|\frac{T(t)y-y}{t}\Bigr\|
=\sup_{t\in (0,1]}\Bigl\|\frac{T(t)_{\mid Y}y-y}{t}\Bigr\|_{Y}<\infty
\] 
and 
\[
Y\ni Cy
=\tau_{Y}\text{-}\lim_{t\to 0\rlim}\frac{T(t)_{\mid Y}y-y}{t}
=\tau\text{-}\lim_{t\to 0\rlim}\frac{T(t)y-y}{t}
=Ay
\]
which yields $D(C)\subset (D(A)\cap Y)$ and thus $D(C)\subset D(A_{\mid Y})$. 
For the converse inclusion choose $\lambda>\max(\omega_{0}(T),\omega_{0}(T_{\mid Y}))$ and note that 
\[
R(\lambda,C)y=\int_{0}^{\infty}\e^{-\lambda s}T(s)y\d s=R(\lambda,A)y, \quad y\in Y,
\]
by \prettyref{thm:generator} (e) and part (a). For $x\in D(A_{\mid Y})$ this yields
\[
x=R(\lambda,A)(\lambda-A)x=R(\lambda,C)(\lambda-A)x\in D(C)
\]
and therefore $D(A_{\mid Y})\subset D(C)$.

We have $D(A_{\mid Y})\subset (D(A)\cap Y)$ by definition. Let $x\in D(A)\cap Y$. Then $T(t)x\in Y$ for all $t\geq 0$ 
and 
\[
X\ni Ax
=\tau\text{-}\lim_{t\to 0\rlim}\frac{T(t)x-x}{t}
=\gamma\text{-}\lim_{t\to 0\rlim}\frac{T(t)x-x}{t},
\]
which implies $Ax\in Y$ as $Y$ is sequentially $\gamma$-closed in $X$. Hence we have $x\in D(A_{\mid Y})$ and so 
$(D(A)\cap Y)\subset D(A_{\mid Y})$.
\end{proof}

\begin{cor}\label{cor:domain_part_generator}
Let $(X,\|\cdot\|,\tau)$ be a sequentially complete Saks space
and $(T(t))_{t\geq 0}$ a bi-continuous semigroup on $X$ with generator $(A,D(A))$. 
If $Y$ is a $(T(t))_{t\geq 0}$-invariant sequentially $\gamma$-closed linear subspace of $X$ and 
$Fav(T)=D(A)$, then $Fav(T_{\mid Y})=Fav (T)\cap Y=D(A_{\mid Y})$.
\end{cor}
\begin{proof}
The inclusion $D(A_{\mid Y})\subset Fav(T_{\mid Y})$ always holds and we have $D(A_{\mid Y})=D(A)\cap Y$ 
by \prettyref{prop:domain_part_generator}. Clearly, $Fav(T_{\mid Y})=Fav (T)\cap Y$ holds as well. 
Let $x\in Fav(T_{\mid Y})$. Then $x\in (D(A)\cap Y)=D(A_{\mid Y})$ since $Fav(T)=D(A)$. 
\end{proof}

Next, we present a proposition that extends \cite[Theorem 3.2.3, p.~55]{vanneerven1992} to the bi-continuous setting.

\begin{prop}\label{prop:favard}
Let $(X,\|\cdot\|,\tau)$ be a sequentially complete d-consistent Mazur--Saks space
and $(T(t))_{t\geq 0}$ a bi-continuous semigroup on $X$ with generator $(A,D(A))$.  
Then $Fav(T)=D({A^{\bullet}}')\cap X_{\operatorname{cont}}=D({A^{\bullet}}')\cap X$. 
\end{prop}
\begin{proof}
Due to \prettyref{cor:bi_sun_dual} $(A^{\bullet},D(A^{\bullet}))$ is the generator of the 
$\|\cdot\|_{X'}$-continuous semigroup $(T^{\bullet}(t))_{t\geq 0}$ on $X^{\bullet}$. 
Hence it follows from \cite[Corollary 2.1.5 (b), p.~92]{butzer1967} with $X_{0}^{\ast}=X^{\bullet\bullet}$ and
$({T^{\bullet\bullet}}(t))_{t\geq 0}=({T^{\bullet}}'(t)_{\mid X^{\bullet\bullet}})_{t\geq 0}$
by \prettyref{rem:bi_sun_sun} that $Fav(T^{\bullet\bullet})=D({A^{\bullet}}')$. The definitions of the Favard 
space and of $T^{\bullet\bullet}$ yield that 
\[
Fav(T)\cap X^{\bullet\bullet}=Fav(T^{\bullet\bullet})\cap X
\]
where $X$ is identified with its image $j(X)$ in ${X^{\bullet}}'$ by \prettyref{cor:embed_bi_sun_dual}. 
Since $X_{\operatorname{cont}}= X^{\bullet\bullet}\cap X$ by \prettyref{cor:embed_bi_sun_dual} again and 
$Fav(T)\subset X_{\operatorname{cont}}$ by \prettyref{rem:favard} (b), the statement is proved.
\end{proof}

In the $\bullet$-reflexive resp.~semi-reflexive case
we have the following corollary of \prettyref{prop:favard}, which generalises 
\cite[Corollary 3.2.4, p.~55]{vanneerven1992}.

\begin{cor}\label{cor:favard_reflexive}
Let $(X,\|\cdot\|,\tau)$ be a sequentially complete d-consistent Mazur--Saks space
and $(T(t))_{t\geq 0}$ a bi-continuous semigroup on $X$ with generator $(A,D(A))$.  
If $X$ is $\bullet$-reflexive w.r.t.~$(T(t))_{t\geq 0}$, then $Fav(T)=D({A^{\bullet}}')$. 
If $(X,\gamma)$ is semi-reflexive, then $Fav(T)=D(A)$. 
\end{cor}
\begin{proof}
Since $D({A^{\bullet}}')\subset X^{\bullet\bullet}$ by \prettyref{rem:bi_sun_sun}, 
the first part of our statement follows from \prettyref{prop:favard}.
Let us consider the second part. Let $(X,\gamma)$ be semi-reflexive. Then 
$X$ is $\bullet$-reflexive w.r.t.~$(T(t))_{t\geq 0}$ by \prettyref{prop:relation_reflexivities_1} 
and $X={X^{\bullet}}'$ via the canonical map $j$. Hence we have $Fav(T)=D({A^{\bullet}}')$ by the first 
part of our statement. As $D(A)\subset Fav(T)$ by \prettyref{rem:favard} (a), we only need to prove that 
$D({A^{\bullet}}')\subset D(A)$. Let $\re\lambda >\omega_{0}(T)$.
Then it follows from \prettyref{thm:resolv_bullet} (c) and \prettyref{prop:resolv_bullet} (a) that
$R(\lambda,A)x=R(\lambda,{A^{\bullet}}')x$ for all $x\in X$. 
Let $y\in D({A^{\bullet}}')$. Then there is ${x^{\bullet}}'\in {X^{\bullet}}'=X$ such that 
$R(\lambda,{A^{\bullet}}'){x^{\bullet}}'=y$ and
\[
y=R(\lambda,{A^{\bullet}}'){x^{\bullet}}'=R(\lambda,A){x^{\bullet}}'\in D(A),
\]
proving $D({A^{\bullet}}')\subset D(A)$.
\end{proof}

Let us turn to a generalisation of \cite[Lemma 3.2.7, p.~57]{vanneerven1992}.

\begin{lem}\label{lem:gamma_closure_resolv}
Let $(X,\|\cdot\|,\tau)$ be a sequentially complete d-consistent Mazur--Saks space
and $(T(t))_{t\geq 0}$ a bi-continuous semigroup on $X$ with generator $(A,D(A))$. 
Then we have 
\[
        \overline{R(\lambda,A)B_{(X,\|\cdot\|^{\bullet})}}^{\gamma}
\subset (R(\lambda,{A^{\bullet}}')B_{{X^{\bullet}}'}\cap X)
\subset \bigcup_{n\in\N}n\overline{R(\lambda,A)B_{(X,\|\cdot\|^{\bullet})}}^{\operatorname{seq-}\gamma}
\]
for all $\lambda\in\rho(A)$ such that $R(\lambda,A)^{\bullet}X^{\bullet}\subset D(A^{\bullet})$, 
where $\overline{R(\lambda,A)B_{(X,\|\cdot\|^{\bullet})}}^{\operatorname{seq-}\gamma}$ is the sequential $\gamma$-closure 
of $R(\lambda,A)B_{(X,\|\cdot\|^{\bullet})}$.
\end{lem}
\begin{proof}
Due to \prettyref{thm:resolv_bullet} (c) it holds 
$\rho({A^{\bullet}}')=\rho(A^{\bullet})$ and $R(\lambda,{A^{\bullet}}')=R(\lambda,A^{\bullet})'$ for all 
$\lambda\in\rho(A^{\bullet})$. Now, let $\lambda\in\rho(A)$ such that 
$R(\lambda,A)^{\bullet}X^{\bullet}\subset D(A^{\bullet})$. Then it follows from 
\prettyref{thm:resolv_bullet} (c) again that $j(R(\lambda,A)x)=R(\lambda,{A^{\bullet}}')j(x)$ for all $x\in X$ 
with the map $j\colon X\to {X^{\bullet}}'$,
$\langle j(x),x^{\bullet}\rangle=\langle x^{\bullet},x\rangle$, from \prettyref{cor:embed_bi_sun_dual}. 
$j$ is an isometry as a map from $(X,\|\cdot\|^{\bullet})$ to $({X^{\bullet}}',\|\cdot\|_{{X^{\bullet}}'})$. 
We deduce that $R(\lambda,A)B_{(X,\|\cdot\|^{\bullet})}\subset (R(\lambda,{A^{\bullet}}')B_{{X^{\bullet}}'}\cap X)$.
Since $B_{{X^{\bullet}}'}$ is $\sigma({X^{\bullet}}',X^{\bullet})$-weakly compact by the Banach--Alaoglu theorem 
and the resolvent $R(\lambda,{A^{\bullet}}')$ is $\sigma({X^{\bullet}}',X^{\bullet})$-continuous, the set 
$R(\lambda,{A^{\bullet}}')B_{{X^{\bullet}}'}$ is 
$\sigma({X^{\bullet}}',X^{\bullet})$-weakly compact as well. 
Further, $j$ as map from $(X,\gamma)$ to $({X^{\bullet}}',\sigma({X^{\bullet}}',X^{\bullet}))$ is continuous 
because $X^{\bullet}\subset X^{\circ}=(X,\gamma)'$ by \prettyref{rem:bi_cont_dual}. Together with the 
$\sigma({X^{\bullet}}',X^{\bullet})$-weak closedness of $R(\lambda,{A^{\bullet}}')B_{{X^{\bullet}}'}$ this implies 
the first inclusion. 

Next, we show that the second inclusion is a consequence of the equation
\begin{align}\label{eq:gamma_closure_resolv_1}
  \frac{1}{t}\int_{0}^{t}T(s)x\d x
&=R(\lambda,A)(\lambda-A)\frac{1}{t}\int_{0}^{t}T(s)x\d x\nonumber\\
&=R(\lambda,A)\Bigl(\frac{\lambda}{t}\int_{0}^{t}T(s)x\d x-\frac{1}{t}(T(t)x-x)\Bigr),
\end{align}
for all $t>0$ and $x\in X$, which we get from \prettyref{thm:generator} (d). 
Indeed, take $x\in R(\lambda,{A^{\bullet}}')B_{{X^{\bullet}}'}\cap X$. 
Due to \prettyref{prop:favard} we have 
\begin{equation}\label{eq:gamma_closure_resolv_2}
       (R(\lambda,{A^{\bullet}}')B_{{X^{\bullet}}'}\cap X)
\subset (D({A^{\bullet}}')\cap X)
=       Fav(T).
\end{equation}
So, since $x\in Fav(T)$, $(T(t))_{t\geq 0}$ is exponentially bounded, 
$\|\cdot\|=\sup_{p_{\gamma}\in\mathcal{P}_{\gamma}}p_{\gamma}$ on $X$ for some directed system of 
seminorms $\mathcal{P}_{\gamma}$ that generates $\gamma$, and 
$\|\cdot\|^{\bullet}$ is equivalent to $\|\cdot\|$ by \prettyref{thm:eqiv_norms}, the right-hand side 
of \eqref{eq:gamma_closure_resolv_1} remains $\|\cdot\|^{\bullet}$-bounded as $t\to 0\rlim$ whereas 
the left-hand side $\gamma$-converges to $x$ (as a sequence with $t=t_{n}$ for any $(t_{n})_{n\in\N}$ with $t_{n}\to 0\rlim$) 
by the proof of \prettyref{thm:eqiv_norms}. 
Thus there is $n\in\N$ such that $x\in n\overline{R(\lambda,A)B_{(X,\|\cdot\|^{\bullet})}}^{\operatorname{seq-}\gamma}$.
\end{proof}

Due to the equivalence of $\|\cdot\|^{\bullet}$ and $\|\cdot\|$ there is $M\geq 0$ such that 
$B_{\|\cdot\|}\subset B_{(X,\|\cdot\|^{\bullet})}\subset MB_{\|\cdot\|}$, which yields that the lemma above is 
still valid if $\|\cdot\|^{\bullet}$ is replaced by $\|\cdot\|$. 
The next theorem is a generalisation of \cite[Theorem 3.2.8, p.~57]{vanneerven1992} and describes the space
$Fav(T)$ in terms of approximation by elements of $D(A)$.

\begin{thm}\label{thm:favard}
Let $(X,\|\cdot\|,\tau)$ be a sequentially complete d-consistent Mazur--Saks space 
and $(T(t))_{t\geq 0}$ a bi-continuous semigroup on $X$ with generator $(A,D(A))$.
Then the following assertions are equivalent for $x\in X$:
\begin{enumerate}
\item[(i)] $x\in Fav(T)$
\item[(ii)] For some (all) $\lambda\in\rho(A)$ such that $R(\lambda,A)^{\bullet}X^{\bullet}\subset D(A^{\bullet})$ 
there exists a $\|\cdot\|$-bounded sequence $(y_{n})_{n\in\N}$ in $X$ with
$\gamma$-$\lim_{n\to\infty}R(\lambda,A)y_{n}=x$. 
\item[(iii)] For some (all) $\lambda\in\rho(A)$ such that $R(\lambda,A)^{\bullet}X^{\bullet}\subset D(A^{\bullet})$ 
there exist a $\|\cdot\|$-bounded sequence $(y_{n})_{n\in\N}$ in $X$ and $k\in\N_{0}$ 
with $\gamma$-$\lim_{n\to\infty}R(\lambda,A)^{k+1}y_{n}=R(\lambda,A)^{k}x$.
\end{enumerate}
\end{thm}
\begin{proof}
(i)$\Rightarrow$(ii) Let $x\in Fav(T)$ and $\lambda\in\rho(A)$ such that 
$R(\lambda,A)^{\bullet}X^{\bullet}\subset D(A^{\bullet})$. Since $\lambda\in\rho({A^{\bullet}}')$ 
by \prettyref{thm:resolv_bullet} (a) and (c), and
\begin{equation}\label{eq:favard_1}
 (R(\lambda,{A^{\bullet}}'){X^{\bullet}}'\cap X)
=(D({A^{\bullet}}')\cap X)
=Fav(T)
\end{equation}
by \prettyref{prop:favard}, there is $m\in\N$ such that 
$x\in R(\lambda,{A^{\bullet}}')mB_{{X^{\bullet}}'}\cap X$. 
Due to the second inclusion of \prettyref{lem:gamma_closure_resolv} there is $n\in\N$ with 
$x\in mn\overline{R(\lambda,A)B_{(X,\|\cdot\|^{\bullet})}}^{\operatorname{seq-}\gamma}$, confirming the first implication. 
The implication (ii)$\Rightarrow$(iii) is trivial. 

(iii)$\Rightarrow$(i) Let there exist a $\|\cdot\|$-bounded sequence $(y_{n})_{n\in\N}$ in $X$ and $k\in\N_{0}$ 
for which we have that $\gamma$-$\lim_{n\to\infty}R(\lambda,A)^{k+1}y_{n}=R(\lambda,A)^{k}x$ for some 
$\lambda\in\rho(A)$ such that $R(\lambda,A)^{\bullet}X^{\bullet}\subset D(A^{\bullet})$. If $k=0$, then (i) 
is implied by the first inclusion of \prettyref{lem:gamma_closure_resolv} and \eqref{eq:gamma_closure_resolv_2}.
Suppose that $k>0$. Using (iii), \prettyref{thm:resolv_bullet} (a) and 
$X^{\bullet}\subset X^{\circ}=(X,\gamma)'$ by \prettyref{rem:bi_cont_dual}, we obtain 
\begin{align}\label{eq:favard_2}
  \lim_{n\to\infty}\langle R(\lambda,A^{\bullet})x^{\bullet},R(\lambda,A)^{k}y_{n}\rangle
&=\lim_{n\to\infty}\langle R(\lambda,A)^{\bullet}x^{\bullet},R(\lambda,A)^{k}y_{n}\rangle\nonumber\\
&=\lim_{n\to\infty}\langle x^{\bullet},R(\lambda,A)^{k+1}y_{n}\rangle
 =\langle x^{\bullet},R(\lambda,A)^{k}x\rangle\nonumber\\
&=\langle R(\lambda,A^{\bullet})x^{\bullet},R(\lambda,A)^{k-1}x\rangle
\end{align}
for all $x^{\bullet}\in X^{\bullet}$. By \prettyref{cor:bi_sun_dual} we know that 
$R(\lambda,A^{\bullet})X^{\bullet}=D(A^{\bullet})$ and that $(A^{\bullet},D(A^{\bullet}))$ is the generator of 
a $\|\cdot\|_{X'}$-strongly continuous semigroup on $X^{\bullet}$. Thus $D(A^{\bullet})$ is $\|\cdot\|_{X'}$-dense
in $X^{\bullet}$. Let $x^{\bullet}\in X^{\bullet}$. Then there is a sequence $(z_{m}^{\bullet})_{m\in\N}$ 
in $X^{\bullet}$ such that $(R(\lambda,A^{\bullet})z_{m}^{\bullet})_{m\in\N}$ converges to $x^{\bullet}$ 
w.r.t.~$\|\cdot\|_{X'}$. We note that 
\begin{align*}
     |\langle x^{\bullet},R(\lambda,A)^{k}y_{n}-R(\lambda,A)^{k-1}x\rangle|
&\leq\|x^{\bullet}-R(\lambda,A^{\bullet})z_{m}^{\bullet}\|_{X'}
                 \|R(\lambda,A)^{k}y_{n}-R(\lambda,A)^{k-1}x\|\\
&\phantom{\leq}+|\langle R(\lambda,A^{\bullet})z_{m}^{\bullet},R(\lambda,A)^{k}y_{n}-R(\lambda,A)^{k-1}x\rangle|
\end{align*}
for all $n,m\in\N$. 
Since $(R(\lambda,A)^{k}y_{n}-R(\lambda,A)^{k-1}x)_{n\in\N}$ is $\|\cdot\|$-bounded 
by the $\|\cdot\|$-boundedness of $(y_{n})_{n\in\N}$, there 
is $C>0$ such that $\|R(\lambda,A)^{k}y_{n}-R(\lambda,A)^{k-1}x\|\leq C$ for all $n\in\N$. Due to 
$\|\cdot\|_{X'}$-$\lim_{m\to\infty}R(\lambda,A^{\bullet})z_{m}^{\bullet}=x^{\bullet}$, for any $\varepsilon>0$ 
there is $M_{0}\in\N$ such that $\|x^{\bullet}-R(\lambda,A^{\bullet})z_{m}^{\bullet}\|_{X'}\leq \tfrac{\varepsilon}{2C}$ 
for all $m\geq M_{0}$. Then there is $N\in\N$ such that 
$|\langle R(\lambda,A^{\bullet})z_{M_{0}}^{\bullet},R(\lambda,A)^{k}y_{n}-R(\lambda,A)^{k-1}x\rangle|\leq
\tfrac{\varepsilon}{2}$ for all $n\geq N$ by \eqref{eq:favard_2},
which implies that 
\[
    |\langle x^{\bullet},R(\lambda,A)^{k}y_{n}-R(\lambda,A)^{k-1}x\rangle|
\leq\frac{\varepsilon}{2C}C+ \frac{\varepsilon}{2}
=\varepsilon
\]
for all $n\geq N$. 
Thus we have 
\[
\lim_{n\to\infty}\langle x^{\bullet},R(\lambda,A)^{k}y_{n}\rangle
=\langle x^{\bullet},R(\lambda,A)^{k-1}x\rangle
\]
for all $x^{\bullet}\in X^{\bullet}$, which means that $R(\lambda,A)^{k}y_{n}\to R(\lambda,A)^{k-1}x$ in the 
$\sigma(X,X^{\bullet})$-topology. Repeating this argument yields $R(\lambda,A)y_{n}\to x$ in the 
$\sigma(X,X^{\bullet})$-topology. Therefore $x$ is an element of the $\sigma(X,X^{\bullet})$-closure of 
$K R(\lambda,A)B_{\|\cdot\|}$ for some $K\geq 0$ by the $\|\cdot\|$-boundedness of $(y_{n})_{n\in\N}$. 
Due to \prettyref{prop:T(t)_equicont} and \prettyref{rem:T(t)_equicont} (a) $K R(\lambda,A)B_{\|\cdot\|}$  
is $\sigma(X,X^{\circ})$-$(T(t))_{t\geq 0}$-equicontinuous and hence we get 
\[
 \overline{K R(\lambda,A)B_{\|\cdot\|}}^{\sigma(X,X^{\bullet})}
=\overline{K R(\lambda,A)B_{\|\cdot\|}}^{\sigma(X,X^{\circ})}
\]
by \prettyref{cor:weakly_bi_sun_closed}. Since $K R(\lambda,A)B_{\|\cdot\|}$ is convex and 
$(X,\sigma(X,X^{\circ}))'=X^{\circ}=(X,\gamma)'$, we obtain 
\[
 \overline{K R(\lambda,A)B_{\|\cdot\|}}^{\sigma(X,X^{\bullet})}
=\overline{K R(\lambda,A)B_{\|\cdot\|}}^{\sigma(X,X^{\circ})}
=\overline{K R(\lambda,A)B_{\|\cdot\|}}^{\gamma}
=K\overline{R(\lambda,A)B_{\|\cdot\|}}^{\gamma}
\]
by \cite[8.2.5 Proposition, p.~149]{jarchow1981}. 
In combination with the first inclusion of \prettyref{lem:gamma_closure_resolv} and \eqref{eq:favard_1} we conclude 
that $x\in Fav(T)$.
\end{proof}

Our next result generalises \cite[Theorem 3.2.9, p.~57]{vanneerven1992}. 

\begin{thm}\label{thm:favard_gamma_closed_resolv}
Let $(X,\|\cdot\|,\tau)$ be a sequentially complete d-consistent Mazur--Saks space 
and $(T(t))_{t\geq 0}$ a bi-continuous semigroup on $X$ with generator $(A,D(A))$. 
Then the following assertions are equivalent:
\begin{enumerate}
\item[(i)] $Fav(T)=D(A)$
\item[(ii)] $R(\lambda,A)B_{(X,\|\cdot\|^{\bullet})}$ is $\gamma$-closed 
for some (all) $\lambda\in\rho(A)$ such that $R(\lambda,A)^{\bullet}X^{\bullet}\subset D(A^{\bullet})$.
\item[(iii)] $R(\lambda,A)B_{(X,\|\cdot\|^{\bullet})}$ is sequentially $\gamma$-closed 
for some (all) $\lambda\in\rho(A)$ such that $R(\lambda,A)^{\bullet}X^{\bullet}\subset D(A^{\bullet})$.
\end{enumerate}
\end{thm}
\begin{proof}
(i)$\Rightarrow$(ii) Suppose that $Fav(T)=D(A)$. 
Let $\lambda\in\rho(A)$ such that $R(\lambda,A)^{\bullet}X^{\bullet}\subset D(A^{\bullet})$
and $y\in\overline{R(\lambda,A)B_{(X,\|\cdot\|^{\bullet})} }^{\gamma}$. 
By the first inclusion of \prettyref{lem:gamma_closure_resolv} there is ${x^{\bullet}}'\in B_{{X^{\bullet}}'}$ 
such that $j(y)=R(\lambda,{A^{\bullet}}'){x^{\bullet}}'$. If follows from \eqref{eq:favard_1} 
that $y\in Fav(T)$ and hence by our assumption that there is $x\in X$ such that 
$R(\lambda,A)x=y$. Due to \prettyref{thm:resolv_bullet} (c) we have $j(R(\lambda,A)x)=R(\lambda,{A^{\bullet}}')j(x)$ 
and the injectivity of $R(\lambda,{A^{\bullet}}')$ yields $j(x)={x^{\bullet}}'$. 
But $j$ is an isometry as a map from $(X,\|\cdot\|^{\bullet})$ to $({X^{\bullet}}',\|\cdot\|_{{X^{\bullet}}'})$, 
which implies $x\in B_{(X,\|\cdot\|^{\bullet})}$. Therefore $y\in R(\lambda,A)B_{(X,\|\cdot\|^{\bullet})}$, 
meaning that $R(\lambda,A)B_{(X,\|\cdot\|^{\bullet})}$ is $\gamma$-closed. This proves the first implication.
The implication (ii)$\Rightarrow$(iii) is trivial. 

(iii)$\Rightarrow$(i) Now, suppose that $R(\lambda,A)B_{(X,\|\cdot\|^{\bullet})}$ is sequentially $\gamma$-closed 
for some $\lambda\in\rho(A)$ such that $R(\lambda,A)^{\bullet}X^{\bullet}\subset D(A^{\bullet})$.
Then we derive from 
\[
\overline{R(\lambda,A)B_{(X,\|\cdot\|^{\bullet})}}^{\operatorname{seq-}\gamma}
=R(\lambda,A)B_{(X,\|\cdot\|^{\bullet})}
\]
and the second inclusion of \prettyref{lem:gamma_closure_resolv} that 
\[
(R(\lambda,{A^{\bullet}}')B_{{X^{\bullet}}'}\cap X)
\subset\bigcup_{n\in\N}nR(\lambda,A)B_{(X,\|\cdot\|^{\bullet})}
=D(A).
\]
This gives us $Fav(T)\subset D(A)$ by \eqref{eq:favard_1}. The converse inclusion is also true by 
\prettyref{rem:favard} (a).
\end{proof}

\begin{rem}
We note that we may replace the (sequential) $\gamma$-closures in 
\prettyref{lem:gamma_closure_resolv} and the (sequential) $\gamma$-closedness in \prettyref{thm:favard_gamma_closed_resolv} 
as well as the $\gamma$-limits in \prettyref{thm:favard} by (sequential) $\tau$-closures, (sequential) $\tau$-closedness 
and $\tau$-limits, respectively, by \prettyref{defn:mixed_top_Saks} (a) and \cite[I.1.10 Proposition, p.~9]{cooper1978}.
\end{rem}

In the $\bullet$-reflexive case we have the following generalisation of 
\cite[Theorem 3.2.12, p.~59]{vanneerven1992}.

\begin{thm}\label{thm:bi_sun_reflexive_j_surjective}
Let $(X,\|\cdot\|,\tau)$ be a sequentially complete d-consistent Mazur--Saks space 
and $(T(t))_{t\geq 0}$ a bi-continuous semigroup on $X$ with generator $(A,D(A))$. 
Suppose that $X$ is $\bullet$-reflexive w.r.t.~$(T(t))_{t\geq 0}$. 
Then the following assertions are equivalent:
\begin{enumerate}
\item[(i)] $j\colon X\to {X^{\bullet}}'$ is surjective. 
\item[(ii)] $R(\lambda,A)B_{(X,\|\cdot\|^{\bullet})}$ is $\sigma(X,X^{\circ})$-compact 
for some (all) $\lambda\in\rho(A)$ such that $R(\lambda,A)^{\bullet}X^{\bullet}\subset D(A^{\bullet})$.
\item[(iii)] $R(\lambda,A)B_{(X,\|\cdot\|^{\bullet})}$ is $\sigma(X,X^{\bullet})$-compact 
for some (all) $\lambda\in\rho(A)$ such that $R(\lambda,A)^{\bullet}X^{\bullet}\subset D(A^{\bullet})$.
\item[(iv)] $B_{(X,\|\cdot\|^{\bullet})}$ is $\sigma(X,X^{\bullet})$-compact.
\end{enumerate}
Each of the assertions (i)-(iv) implies $Fav(T)=D(A)$.
\end{thm}
\begin{proof}
(ii)$\Rightarrow$(i) Condition (ii) implies that $R(\lambda,A)B_{(X,\|\cdot\|^{\bullet})}$ is $\gamma$-closed
for some $\lambda\in\rho(A)$ such that $R(\lambda,A)^{\bullet}X^{\bullet}\subset D(A^{\bullet})$ 
because $(X,\gamma)'=X^{\circ}$ and the $\sigma(X,X^{\circ})$-topology is coarser than $\gamma$. 
By \prettyref{thm:favard_gamma_closed_resolv} we obtain $Fav(T)=D(A)$ and from the $\bullet$-reflexivity of $X$ 
we derive $D({A^{\bullet}}')\subset X^{\bullet\bullet}=X_{\operatorname{cont}}$. This implies 
\[
D({A^{\bullet}}')=D({A^{\bullet}}')\cap X_{\operatorname{cont}}=Fav(T)=D(A)
\]
by \prettyref{prop:favard} and thus
\[
{X^{\bullet}}'
=(\lambda-{A^{\bullet}}')D({A^{\bullet}}')
=(\lambda-{A^{\bullet}}')D(A)
=(\lambda-A)D(A)
=X,
\]
yielding the desired result.

(i)$\Rightarrow$(iv) $B_{{X^{\bullet}}'}$ is $\sigma({X^{\bullet}}',X^{\bullet})$-compact by the Banach--Alaoglu 
theorem. By assumption we may identify $X$ and ${X^{\bullet}}'$ as well as 
$B_{{X^{\bullet}}'}$ and $B_{(X,\|\cdot\|^{\bullet})}$ via $j$ because $j$ is an isometry as a map from 
$(X,\|\cdot\|^{\bullet})$ to $({X^{\bullet}}',\|\cdot\|_{{X^{\bullet}}'})$. 

(ii)$\Leftrightarrow$(iii) $R(\lambda,A)B_{(X,\|\cdot\|^{\bullet})}$ is $\gamma$-$(T(t))_{t\geq 0}$-equicontinuous 
by \prettyref{prop:resolv_equicont} and \prettyref{thm:eqiv_norms} and thus 
$\sigma(X,X^{\circ})$-$(T(t))_{t\geq 0}$-equicontinuous by \prettyref{rem:T(t)_equicont} (a).
Due to \prettyref{cor:weakly_bi_sun_relative_closed} the relative $\sigma(X,X^{\circ})$- and 
$\sigma(X,X^{\bullet})$-topology coincide on $R(\lambda,A)B_{(X,\|\cdot\|^{\bullet})}$, which implies 
the validity of the equivalence (ii)$\Leftrightarrow$(iii).

(iv)$\Rightarrow$(ii) This follows from \prettyref{prop:resolvent_weakly_compact} (b).
\end{proof} 

\begin{exa}
Let $q\colon \N\to\C$ such that $\sup_{n\in\N}\re q(n)<\infty$, and let $(T(t))_{t\geq 0}$ be the 
bi-continuous multiplication semigroup on $(\ell^{\infty},\|\cdot\|_{\infty},\mu(\ell^{\infty},\ell^{1}))$ from 
\prettyref{ex:resolv_bullet} given by 
\[
T(t)x\coloneqq(\e^{q(n)t}x_{n})_{n\in\N},\quad x\in\ell^{\infty},\, t\geq 0,
\]
with generator
$
A\colon D(A)\to \ell^{\infty},\;Ax =qx,
$
on the domain 
\[
D(A)=\{x\in\ell^{\infty}\;|\;(q(n)x_{n})_{n\in\N}\in\ell^{\infty}\}. 
\]
Furthermore, it holds
\[
\|T(t)\|_{\mathcal{L}(\ell^{\infty})}=\e^{t\sup_{n\in\N}\re q(n)},\quad t\geq 0,
\]
which implies $\omega_{0}(T)=\sup_{n\in\N}\re q(n)$ and 
$M\coloneqq\limsup_{t\to 0\rlim}\|T(t)\|_{\mathcal{L}(\ell^{\infty})}=1$. 
Therefore $\|\cdot\|_{\infty}=\|\cdot\|_{\infty}^{\bullet}$ by \prettyref{cor:embed_bi_sun_dual}.
The space $(\ell^{\infty},\mu(\ell^{\infty},\ell^{1}))$ is a semi-reflexive Mackey--Mazur space, in particular 
$\ell^{\infty}$ is $\bullet$-reflexive w.r.t.~$(T(t))_{t\geq 0}$, 
and $j\colon \ell^{\infty}\to {(\ell^{\infty})^{\bullet}}'$ is surjective 
by \prettyref{cor:relation_reflexivities} and \prettyref{prop:relation_reflexivities_1}. 
It follows from \prettyref{ex:resolv_bullet} and \prettyref{thm:bi_sun_reflexive_j_surjective} 
that $Fav(T)=D(A)$ and $R(\lambda,A)B_{(\ell^{\infty},\|\cdot\|_{\infty}^{\bullet})}$ is 
$\sigma(\ell^{\infty},\ell^{1})$-compact 
for all $\lambda\in\rho(A)=\C\setminus\overline{q(\N)}$.
\end{exa}

Of course, instead of the surjectivity of $j\colon \ell^{\infty}\to {(\ell^{\infty})^{\bullet}}'$ one can also
use in the example above that $B_{(\ell^{\infty},\|\cdot\|_{\infty})}$ is $\sigma(\ell^{\infty},\ell^{1})$-compact 
by the Banach--Alaoglu theorem and that $\|\cdot\|_{\infty}=\|\cdot\|_{\infty}^{\bullet}$ to conclude that 
$R(\lambda,A)B_{(\ell^{\infty},\|\cdot\|_{\infty}^{\bullet})}$ is $\sigma(\ell^{\infty},\ell^{1})$-compact 
for all $\lambda\in\rho(A)$ and $Fav(T)=D(A)$ by \prettyref{thm:bi_sun_reflexive_j_surjective}. 
Another way to prove $Fav(T)=D(A)$ 
by \prettyref{ex:favard_dual_semigroup} is to observe that $(T(t))_{t\geq 0}$ is the dual semigroup of 
the $\|\cdot\|_{\ell^{1}}$-strongly continuous multiplication semigroup $(S(t))_{t\geq 0}$ on $\ell^{1}$ given by 
$S(t)x\coloneqq(\e^{q(n)t}x_{n})_{n\in\N}$ for $x\in\ell^{1}$ and $t\geq 0$.

\bibliography{biblio_sun_dual_theory}
\bibliographystyle{plainnat}
\end{document}